\documentclass{daj}

\usepackage{amsfonts}
\usepackage{amssymb}
\usepackage{amsthm}
\usepackage{amsmath}
\usepackage{amscd}
\usepackage[capitalise]{cleveref}

\newtheorem{prop}{Proposition}[section]
\newtheorem{theorem}[prop]{Theorem}
\newtheorem{lemma}[prop]{Lemma}
\newtheorem{corollary}[prop]{Corollary}

\theoremstyle{definition}
\newtheorem{definition}[prop]{Definition}

\theoremstyle{remark}

\newtheorem*{remark*}{Remark}
\newtheorem*{remarks*}{Remarks}
\newtheorem{remark}[prop]{Remark}

\newcommand{\R}{\mathbb{R}}
\newcommand{\C}{\mathbb{C}}
\newcommand{\N}{\mathbb{N}}
\newcommand{\Z}{\mathbb{Z}}
\newcommand{\Q}{\mathbb{Q}}
\newcommand{\ord}{\text{\textup{ord}}}
\newcommand{\cD}{{\lceil D\rceil}}
\newcommand{\eps}{\varepsilon}
\newcommand{\g}{\mathfrak{g}}
\newcommand{\Span}{\text{\textup{Span}}}
\newcommand*{\vol}{\mathop{\textup{vol}}\nolimits}
\newcommand{\Lie}{\text{\textup{Lie}}}
\newcommand{\z}{\mathfrak{z}}

\numberwithin{equation}{section}

\dajAUTHORdetails{%
  title = {Properness of Nilprogressions and the Persistence of Polynomial Growth of Given Degree}, 
  author = {Romain Tessera and Matthew C. H. Tointon},
  plaintextauthor = {Romain Tessera, Matthew C. H. Tointon},
  copyrightauthor = {Romain Tessera and Matthew Tointon},
}   

\dajEDITORdetails{%
   year={2018},
   number={17},
   received={11 January 2017},   
   published={6 November 2018},  
   doi={10.19086/da.5056},       
}   

\begin{document}

\begin{frontmatter}[classification=text]


\author[romain]{Romain Tessera\thanks{Supported by ANR Project GAMME (ANR-14-CE25-0004)}}
\author[matt]{Matthew C. H. Tointon\thanks{Supported by ERC grant GA617129 `GeTeMo' for part of the work of the project, and by a Junior Research Fellowship from Homerton College, University of Cambridge for the rest}}

\begin{abstract}
We show that an arbitrary nilprogression can be approximated by a proper coset nilprogression in upper-triangular form. This can be thought of as a nilpotent version of the Freiman--Bilu result that a generalised arithmetic progression can be efficiently contained in a proper generalised arithmetic progression, and indeed an important ingredient in the proof is a Lie-algebra version of the geometry-of-numbers argument at the centre of that result. We also present some applications. We verify a conjecture of Benjamini that if $S$ is a symmetric generating set for a group such that $1\in S$ and $|S^n|\le Mn^D$ at some sufficiently large scale $n$ then $S$ exhibits polynomial growth of the same degree $D$ at all subsequent scales, in the sense that $|S^r|\ll_{M,D}r^D$ for every $r\ge n$. Our methods also provide an important ingredient in a forthcoming companion paper in which we reprove and sharpen a result about scaling limits of vertex-transitive graphs of polynomial growth due to Benjamini, Finucane and the first author. We also note that our arguments imply that every approximate group has a large subset with a large quotient that is Freiman isomorphic to a subset of a torsion-free nilpotent group of bounded rank and step.
\end{abstract}
\end{frontmatter}

\section{Introduction}
\noindent\textsc{Background.} A finite subset $A$ of a group $G$ is said to have \emph{doubling} at most $K>0$ if $|A^2|\le K|A|$; it is said to be a \emph{$K$-approximate subgroup of $G$}, or simply a \emph{$K$-approximate group}, if it is symmetric and contains the identity and there exists $X\subset G$ with $|X|\le K$ such that $A^2\subset XA$. Here, and throughout this paper, we use the standard notation $AB=\{ab:a\in A,b\in B\}$, $A^n=\{a_1\cdots a_n:a_i\in A\}$ and $A^{-n}=\{a_1^{-1}\cdots a_n^{-1}:a_i\in A\}$. In the foundational work \cite{tao.product.set}, Tao shows that for many practical purposes sets of bounded doubling and approximate groups are essentially interchangeable.

In recent years there has been a large body of work studying approximate groups and applying them in an impressive array of fields. We refer the reader to the surveys \cite{bgt.survey,app.grps,ben.icm,helf.survey,sand.survey} for further background, as well as details of some of these applications.

An important aim of approximate group theory is to describe the algebraic structure of approximate groups, and there is a very general result of Breuillard, Green and Tao \cite{bgt} in this direction. Before we can state this result we need some definitions. Let $u_1,\ldots,u_r$ be elements of a group $G$ and let $L=(L_1,\ldots,L_r)$ be a vector of positive integers. The set of all products in the $u_i$ and their inverses in which each $u_i$ and its inverse appear at most $L_i$ times between them is called a \emph{progression} of rank $r$ and side lengths $L_1,\ldots,L_r$, and is denoted $P^\ast(u_1,\ldots,u_r;L_1,\ldots,L_r)$. We abbreviate this variously to $P^\ast(u_1,\ldots,u_r;L)$ and $P^\ast(u,L)$.

If $H$ is a finite subgroup of $G$ that is normalised by $P^\ast(u,L)$, and if $u_1,\ldots,u_r$ generate an $s$-step nilpotent group modulo $H$, then $HP^\ast(u,L)$ is said to be a \emph{coset nilprogression} of \emph{rank} $r$ and \emph{step} $s$. If $H$ is trivial, we say simply that $P^\ast(u,L)$ is a \emph{nilprogression} of \emph{rank} $r$ and \emph{step} $s$.

Broadly speaking, the result of Breuillard, Green and Tao is then as follows.
\begin{theorem}[Breuillard--Green--Tao {\cite[Corollary 2.11]{bgt}}, partial statement]\label{thm:bgt.partial}
Let $A$ be a $K$-approximate group. Then there exists a coset nilprogression $HP\subset A^4$ of rank and step at most $O_K(1)$ and a set $X\subset\langle A\rangle$ with $|X|\ll_K1$ such that $A\subset XHP$.
\end{theorem}

An important way in which this result could be improved is that, as things stand, the bound on the size of $X$ is not effective. This is essentially due to the use of a non-principal ultrafilter in the proof.

There are a number of results due to various authors that remove this ineffectiveness in return for restricting to cases in which $A$ generates certain particular classes of group, such as abelian groups \cite{freiman,green-ruzsa,ruzsa.Z,ruzsa}, residually nilpotent groups \cite{nfdl,resid}, soluble groups \cite{tao.solv}, or linear groups or groups of Lie type \cite{sol.lin,unitary,bgt.lin,bgt.lin.note,gill-helf,helfgott1,helfgott2,pyb-sza,pyb-sza.2}. In the case that $A$ generates a nilpotent group, we have the following Freiman-type result of the second author.
\begin{theorem}[{\cite[Theorem 1.5]{nilp.frei}}]\label{thm:nilp.Frei}
Let $A$ be a $K$-approximate group such that $\langle A\rangle$ is $s$-step nilpotent. Then there exists a nilpotent coset progression $HP\subset A^{K^{O_s(1)}}$ of rank at most $K^{O_s(1)}$ such that $A\subset HP$.
\end{theorem}
We define the term \emph{nilpotent coset progression} in Section \ref{sec:nilp-progs}. A nilpotent coset progression is an object strongly analogous to a coset nilprogression -- indeed, it is shown in \cite[Proposition C.1]{nilp.frei} that the two are essentially interchangable in the context of approximate groups -- and so it will do little harm for the reader to substitute mentally `coset nilprogression' for `nilpotent coset progression' in Theorem \ref{thm:nilp.Frei} and throughout this introduction.

One can essentially reduce Theorem \ref{thm:bgt.partial} to Theorem \ref{thm:nilp.Frei} by proving the following intermediate result.
\begin{theorem}[Breuillard--Green--Tao, simple form]\label{thm:bgt.noprog}
Let $A$ be a $K$-approximate group. Then there exists a group $\Gamma<\langle A\rangle$ with normal subgroup $H\subset A^4$ such that $\Gamma/H$ is nilpotent of step at most $O_K(1)$, and a $K^3$-approximate group $B\subset A^2\cap\Gamma$ and a set $X\subset\langle A\rangle$ with $|X|\ll_K1$ such that $A\subset XB$.
\end{theorem}
\begin{proof}
This follows from \cite[Theorem 1.6]{bgt} and \cite[Lemmas 2.2 \& 2.3]{resid}.
\end{proof}
Indeed, this is precisely the approach taken to proving Theorem \ref{thm:bgt.partial} in Breuillard's lecture notes \cite{eb.minnesota}. (Note, however, that Theorem \ref{thm:bgt.partial} predates Theorem \ref{thm:nilp.Frei}, and so the original proof of Theorem \ref{thm:bgt.partial} was necessarily via a different method, and in particular implies a version of Theorem \ref{thm:nilp.Frei}, albeit one with far worse bounds.) The papers \cite{sol.lin,bgt.lin,nfdl,gill-helf,resid} also all essentially prove cases of Theorem \ref{thm:bgt.partial} by first proving effective versions of Theorem \ref{thm:bgt.noprog}, and then applying Theorem \ref{thm:nilp.Frei} (or an earlier partial result of Breuillard and Green \cite{bg} valid when $G$ is torsion-free). However, whilst this method is essentially sufficient to prove Theorem \ref{thm:bgt.partial} as stated above, there is a more detailed version of Theorem \ref{thm:bgt.partial} that contains more refined information and does not follow from Theorems \ref{thm:nilp.Frei} and \ref{thm:bgt.noprog}.

Before we can give this more detailed version of Theorem \ref{thm:bgt.partial} we need some further definitions. First, following \cite{nilp.frei}, we define the \emph{ordered progression} on generators $u_1,\ldots,u_d\in G$ with lengths $L_1,\ldots,L_d$ to be
\[
P_\ord(u;L):=\{u_1^{\ell_1}\cdots u_d^{\ell_d}:|l_i|\le L_i\}.
\]
If $P$ is an ordered progression and $H$ is a finite subgroup normalised by $P$, then we say that $HP$ is an \emph{ordered coset progression}.

Following \cite{bgt}, we say that the tuple $(u;L)=(u_1,\ldots,u_d;L_1,\ldots,L_d)$ is in \emph{$C$-upper-triangular form} if, whenever $1\le i<j\le d$, for all four choices of signs $\pm$ we have
\begin{equation}\label{eq:C-upp-tri}
[u_i^{\pm1},u_j^{\pm1}]\in P_\ord\left(u_{j+1},\ldots,u_d;\textstyle{\frac{CL_{j+1}}{L_iL_j},\ldots,\frac{CL_d}{L_iL_j}}\right).
\end{equation}
We say that a nilprogression or ordered progression is in $C$-upper-triangular form if the corresponding tuple is. We say that a coset nilprogression or ordered coset progression $HP$ is in $C$-upper-triangular form if the corresponding tuple is in $C$-upper-triangular form modulo $H$.
\begin{remark*}
In \cite{bgt} the definition of $C$-upper-triangular form in fact requires only that
\[
[u_i^{\pm1},u_j^{\pm1}]\in P^\ast\left(u_{j+1},\ldots,u_d;\textstyle{\frac{CL_{j+1}}{L_iL_j},\ldots,\frac{CL_d}{L_iL_j}}\right),
\]
which is weaker than \eqref{eq:C-upp-tri}. However, it is convenient for us to use this slightly more restrictive condition, which of course does not weaken our results at all; indeed, it strengthens them slightly.
\end{remark*}

Given $m>0$, a nilprogression or ordered progression $P$ on the tuple $(u;L)=(u_1,\ldots,u_d;L_1,\ldots,L_d)$ is said to be \emph{$m$-proper with respect to a homomorphism $\pi:\langle P\rangle\to N$} if the elements $\pi(u_1^{\ell_1}\cdots u_d^{\ell_d})$ are all distinct as the $\ell_i$ range over those integers with $|\ell_i|\le mL_i$. The progression $P$ is said to be $m$-proper with respect to a subgroup $H\lhd\langle HP\rangle$ if $P$ is $m$-proper with respect to the quotient homomorphism $\langle HP\rangle\to\langle HP\rangle/H$. In this case we also say that the coset nilprogression or ordered coset progression $HP$ is \emph{$m$-proper}. If a coset nilprogression or ordered coset progression $HP$ is $m$-proper for every $m>0$ then we say it is \emph{infinitely proper}. Note that if $HP$ is $1$-proper then
\begin{equation}\label{eq:proper.size}
|HP|\ge L_1\cdots L_d|H|.
\end{equation}

Having made these definitions, we can now state the more detailed version of Theorem \ref{thm:bgt.partial}, as follows.
\begin{theorem}[Breuillard--Green--Tao {\cite[Corollary 2.11]{bgt}}, complete statement]\label{thm:bgt}
Let $A$ be a $K$-approximate group. Then there exist an $\Omega_K(1)$-proper coset nilprogression $HP\subset A^4$, of rank and step at most $O_K(1)$ and in $O_K(1)$-upper-triangular form, and a set $X\subset\langle A\rangle$ with $|X|\ll_K1$ such that $A\subset XHP$.
\end{theorem}
\begin{remarks*}
The complete statement of \cite[Corollary 2.11]{bgt} also contains a statement about the cardinality of $P$ compared to its side lengths but, as is remarked in \cite{bgt}, this is already essentially implied by Theorem \ref{thm:bgt} as stated above.

Tao \cite[Proposition 3.1]{tao.growth} has shown that given $m>0$ the coset nilprogression in Theorem \ref{thm:bgt} can, at the expense of worsening some of the other implied constants, be taken to be $m$-proper.
\end{remarks*}

Many applications of the results described above do not actually need the full strength of Theorem \ref{thm:bgt}. For example, even Theorem \ref{thm:bgt.noprog} is enough to prove Gromov's polynomial-growth theorem (see \cite[Corollary 11.7]{bgt}). Nonetheless, there are certain applications, such as those of \cite{bt,tao.growth} and Theorem \ref{thm:persitstence.poly.growth}, below, where the properness and upper-triangular form of Theorem \ref{thm:bgt} play a significant role.

\bigskip

\noindent\textsc{Principal new results.} The main purpose of this paper is to obtain properness and upper-triangular form of the nilprogression in Theorem \ref{thm:nilp.Frei}, as follows.

\begin{theorem}\label{thm:nilp.Frei.proper}
Let $A$ be a $K$-approximate group such that $\langle A\rangle$ is $s$-step nilpotent. Then for every $m,C>0$ there exist an $m$-proper ordered coset progression $HP\subset A^{O_{K,m,C}(1)}$, of rank at most $K^{O_s(1)}$ and in $C$-upper-triangular form, and a set $X\subset\langle A\rangle$ with $|X|\ll_{K,s}1$ such that $A\subset XHP$.
\end{theorem}

\begin{remark}Our arguments lead to results expressing approximate groups in terms of ordered progressions in upper-triangular form, rather than nilprogressions. We show in Section \ref{sec:ordered.v.nil} that proper ordered progressions in upper-triangular form always have small doubling, and so it is natural that they should arise in the study of approximate groups. Moreover, all of our results can be converted to be in terms of nilprogressions, in line with the existing literature, since it follows directly from \cite[Proposition C.1]{nilp.frei} that if $(u_1,\ldots,u_d;L_1,\ldots,L_d)$ is in $C$-upper-triangular form then $P_\ord(u;L)\subset P^\ast(u;L)\subset P_\ord(u;L)^{O_d(1)}$.
\end{remark} 

\begin{remark}Theorem \ref{thm:nilp.Frei.proper} is in principle effective, and implies in particular an effective version of Theorem \ref{thm:bgt} for any group that has an effective version of Theorem \ref{thm:bgt.noprog}. It also means that a fully general effective proof of Theorem \ref{thm:bgt.noprog} would yield an effective version of Theorem \ref{thm:bgt} as an immediate corollary.
\end{remark}

\bigskip

\noindent\textsc{Overview of the argument.} When $A\subset\Z$, \cref{thm:nilp.Frei,thm:nilp.Frei.proper} are essentially the classical Freiman--Ruzsa theorem \cite{freiman,ruzsa.Z}. See also Bilu \cite[\S3]{bilu} for a refinement of \cref{thm:nilp.Frei} to Theorem \ref{thm:nilp.Frei.proper} in this case. The approach of this refinement is to take the progression given by \cref{thm:nilp.Frei} and then to transform it into an $m$-proper progression (the upper-triangular form condition is vacuous in the abelian setting).

Let us give a very brief overview of this argument. One starts with the observation that an abelian progression $P$ of rank $r$ is the image of a box $B\subset\Z^r$ under a homomorphism $\pi:\Z^r\to\Z$. If $P$ is not $m$-proper then that means that there is some point $x\in2mB$ such that $\pi(x)=0$, and so $P$ is also the homomorphic image of some lattice convex body $B'\subset\Z^r/\langle x\rangle$. One then uses some geometry of numbers to show that $B'$ can be efficiently contained in a box $B''\subset\Z^r/\langle x\rangle$, and so $P$ is efficiently contained in the homomorphic image of a box of dimension $r-1$, which is by definition a progression of dimension $r-1$. Theorem \ref{thm:nilp.Frei.proper} for $G=\Z$ then follows by induction.

The central strategy of the present paper is to run a similar argument in the nilpotent case. We note in Remark \ref{rem:free.image} that a nilpotent progression is the homomorphic image of a certain type of box in a lattice in a nilpotent Lie group, and then in Sections \ref{sec:geom-of-nos} and \ref{sec:gen.of.progs} we adapt the geometry-of-numbers argument described above to this setting.

This approach ultimately yields the following result, which we prove in Section \ref{sec:Bilu} (in fact, we prove Theorem \ref{thm:nilp.bilu.detailed}, which is a slightly more detailed version of Theorem \ref{thm:nilp.bilu}).
\begin{theorem}\label{thm:nilp.bilu}
Let $P_0$ be a nilpotent progression of rank $r$ and step $s$. Then for every $m,C>0$ there exists an $m$-proper ordered coset progression $HP$, of total rank at most $O_{r,s}(1)$ and in $C$-upper-triangular form, and a set $X\subset\langle P_0\rangle$ with $|X|\ll_{r,s}1$ such that $P_0\subset XHP\subset P_0^{O_{r,s,C,m}(1)}$.
\end{theorem}
\begin{proof}[Proof of Theorem \ref{thm:nilp.Frei.proper}]
Theorem \ref{thm:nilp.Frei.proper} follows immediately from Theorems \ref{thm:nilp.Frei} and \ref{thm:nilp.bilu}.
\end{proof}

\bigskip\noindent\textsc{The abelian case.} In the abelian case our argument yields a slightly stronger statement than Theorem \ref{thm:nilp.Frei.proper}, as follows.

\begin{theorem}\label{thm:ab.Bilu}
Let $A$ be a $K$-approximate group such that $\langle A\rangle$ is abelian. Then for every $m>0$ there exists an $m$-proper coset progression $HP$ of rank at most $K^{O(1)}$ such that 
\[
A\subset HP\subset A^{O_{K,m}(1)}.
\]
\end{theorem}

We prove Theorem \ref{thm:ab.Bilu} at the end of Section \ref{sec:Bilu}. Note that it slightly strengthens the Freiman--Ruzsa theorem stated in \cite{bilu} even in the torsion-free case, since in place of the cardinality bound $|HP|\le O_{K,m}(1)|A|$ we have the qualitatively stronger containment $HP\subset A^{O_{K,m}(1)}$. This last strengthening arises from our use of Proposition \ref{prop:rom.8}. The fact that we are also able to generalise to the setting of groups with torsion is ultimately thanks to Green and Ruzsa, who proved the earliest version of \cref{thm:nilp.Frei} for an arbitrary abelian group \cite{green-ruzsa}.

\bigskip\noindent\textsc{Principal applications.}
One of our main motivations for proving Theorem \ref{thm:nilp.Frei.proper} is that much of the material is useful in a forthcoming paper \cite{scaling.limits} in which we sharpen a result of Benjamini, Finucane and the first author \cite[Theorem 3.2.2]{bft}. That result states that if $(\Gamma_n)$ is a sequence of vertex-transitive graphs with discrete automorphism groups and the balls $B_n$ of radius $n$ in the $\Gamma_n$ satisfy $|B_n|\ll n^D$ then for every sequence $m_n\gg n$ the sequence $(\Gamma_{n},\frac{d_{\Gamma_{n}}}{m_n})$ is relatively compact for the Gromov--Hausdorff topology. It also states that every limit point of $(\Gamma_{n},\frac{d_{\Gamma_{n}}}{m_n})$ is a connected nilpotent  Lie  group  equipped  with  a  left-invariant  Carnot-Caratheodory
metric. In our forthcoming paper \cite{scaling.limits} we give a new proof of this, removing the need to assume that the automorphism groups are discrete and showing moreover that the homogeneous dimension of every limit point is at most $D$.

In the present paper we give a related application to sets of polynomial growth. The Breuillard--Green--Tao proof of Gromov's theorem via Theorem \ref{thm:bgt.noprog} yields as a corollary the following result.
\begin{theorem}[Breuillard--Green--Tao {\cite[Corollary 11.9]{bgt}}]
Given $D>0$ there exists $N=N_D$ such that if $n\ge N$ and $S$ is a finite symmetric generating set for a group $G$ such that
\begin{equation}\label{eq:tao.poly.growth}
|S^n|\le n^D|S|
\end{equation}
then for every $r\ge n$ we have $|S^r|\le r^{O_D(1)}|S|$.
\end{theorem}
Thus, if $S$ exhibits polynomial growth of degree $D$ at some sufficiently large scale $n$ then it exhibits polynomial growth of degree bounded in terms of $D$ at all subsequent scales. Benjamini (private communication) has conjectured that if one replaces \eqref{eq:tao.poly.growth} with the more restrictive condition $|S^n|\le Mn^D$ then one should be able to conclude that $S$ exhibits polynomial growth of the \emph{same} degree $D$ at all subsequent scales. In Section \ref{sec:conj-1} we verify this conjecture, arriving at the following result, in which we write $\{D\}=1-\lfloor D\rfloor$, the fractional part of $D$.
\begin{theorem}[Benjamini's conjecture]\label{thm:persitstence.poly.growth}
Given $M,D>0$ there exists $N=N_\cD$ such that if $n\ge\max\{N,NM,(NM)^\frac{1}{1-\{D\}}\}$ and $S$ is a finite symmetric generating set for a group $G$ such that $1\in S$ and
\begin{equation}\label{eq:persist.poly.growth}
|S^n|\le Mn^D
\end{equation}
then for every $r\ge n$ we have $|S^r|\ll_{M,D}r^D$; indeed $|S^r|\ll_\cD(r/n)^{\lfloor D\rfloor}|S^n|$.
\end{theorem}
Having some constants depending on $\cD$ rather than $D$ might look slightly strange, and in some sense does not contain much information. However, it is a convenient means by which to capture the fact that the bounds are uniform on a bounded range of $D$, which will be useful for an application in another forthcoming paper.

The basic approach to \cref{thm:persitstence.poly.growth}, which is already present in \cite{bt,tao.growth}, is to control the growth of $S^r$ in terms of the growth of a certain nilprogression of bounded rank and step.

\begin{remark*} If $G$ is assumed to be abelian then the weaker assumption \eqref{eq:tao.poly.growth} is enough to draw the same conclusion. To see that \cref{thm:persitstence.poly.growth} does not hold with  \eqref{eq:tao.poly.growth} in general, consider the set
\[
S=\left(\begin{smallmatrix}1 & [-n,n] & [-n^3,n^3]\\0&1&[-n,n]\\0&0&1\end{smallmatrix}\right)
\subset\left(\begin{smallmatrix}1 & \Z & \Z\\0&1&\Z\\0&0&1\end{smallmatrix}\right),
\]
which also appears in \cite[Example 1.11]{tao.growth}. This set $S$ satisfies $|S^n|\ll n^3|S|$ regardless of the choice of $n$, but for any fixed $n$ we have $|S^r|\gg r^4$ as $r\to\infty$. (Note that, although $S$ is not symmetric, this can be fixed by considering the set $S\cup S^{-1}$ in its place; we leave the details to the reader.) More generally, Tao \cite{tao.growth} has studied in some detail the possible subsequent growth of sets satisfying \eqref{eq:tao.poly.growth}. See in particular \cite[Theorem 1.9]{tao.growth} and the examples that follow it for more information.
\end{remark*}

\bigskip
\noindent\textsc{A finitary Lie model theorem for approximate groups.} A key precursor to Theorems \ref{thm:bgt.partial}, \ref{thm:bgt.noprog} and \ref{thm:bgt} was the so-called Lie model theorem of Hrushovski. This appeared originally as \cite[Theorem 4.2]{hrush}, and is also stated as \cite[Theorem 3.10]{bgt}. It can be summarised roughly as saying that, in a suitable limit, if $A$ is an approximate group then there exists an approximate group $L\subset A^4$, a finite set $X\subset\langle A\rangle$ with $|X|\ll1$ such that $A\subset XL$, and a subgroup $H\subset L$ that is normal in $\langle L\rangle$ such that $L/H$ is `well modelled' by a compact neighbourhood of the identity in a Lie group. See \cite[Theorem 3.10]{bgt} and the preceding definitions for a precise statement, and in particular for a clarification of the terms `suitable limit' and `well modelled'.

It may be of interest to note that our proof of Theorem \ref{thm:nilp.Frei.proper} gives a finitary version of this result. A standard framework in additive combinatorics in which to `model' one set in terms of another is the \emph{Freiman homomorphism}. Let $A$ and $B$ be subsets of groups. A map $\varphi:A\to B$ is a \emph{Freiman homomorphism of order $m$} if for every $a_1,\ldots,a_{2m}\in A$ with $a_1\cdots a_m=a_{m+1}\cdots a_{2m}$ we have $\varphi(a_1)\cdots\varphi(a_m)=\varphi(a_{m+1})\cdots\varphi(a_{2m})$. The map $\varphi$ is a \emph{Freiman isomorphism of order $m$} if it is a bijection and both $\varphi$ and $\varphi^{-1}$ are Freiman homomorphisms of order $m$. In additive combinatorics, when one says that a set $A$ is `modelled' by a set $B$, in practice one usually means that $A$ is Freiman isomorphic (of some given order) to $B$. Our Lie model theorem is then as follows.
\begin{corollary}[a finitary Lie model theorem for approximate groups]\label{cor:Lie.model}
Let $A$ be a $K$-approximate group. Then for every $m\in\N$ there exist a set $L\subset A^{O_{K,m}(1)}$, a finite subset $X\subset\langle A\rangle$ with $|X|\ll_K1$ such that $A\subset XL$, and a subgroup $H\subset L$ that is normal in $\langle L\rangle$ such that $L/H$ is Freiman $m$-isomorphic to an infinitely proper ordered progression in $1$-upper triangular form in a torsion-free nilpotent group of rank and step at most $O_K(1)$.
\end{corollary}
In fact, we obtain a slightly more detailed result than this. We define at the beginning of \cref{sec:Bilu} an object that we call a \emph{Lie progression}, which is a homomorphic image of a certain progression in a simply connected nilpotent Lie group (see \cref{def:Lie.prog}). Theorem \ref{thm:nilp.bilu.detailed} then shows that a nilpotent progression can be covered by a few translates of a Lie progression (modulo a `small' subgroup), and, as we explain in \cref{sec:Bilu}, this combines with Theorems \ref{thm:nilp.Frei} and \ref{thm:bgt.noprog} and various other results of this paper to imply in particular \cref{cor:Lie.model}.

\begin{remark}
Corollary \ref{cor:Lie.model} relies on Theorem \ref{thm:bgt.noprog}, and is therefore ineffective in general; specifically, our argument does not give an explicit bound on the size of the set $X$. However, the argument is effective for any class of group for which we have an effective version of Theorem \ref{thm:bgt.noprog}.
\end{remark}

\bigskip
\noindent\textsc{Notation.} We follow the standard convention that if $X,Y$ are real quantities  and $z_1,\ldots,z_k$ are variables or constants then the expressions $X\ll_{z_1,\ldots,z_k}Y$ and $Y\gg_{z_1,\ldots,z_k}X$ each mean that there exists a constant $C>0$ depending only on $z_1,\ldots,z_k$ such that $X$ is always at most $CY$. Moreover, the notation $O_{z_1,\ldots,z_k}(Y)$ denotes a quantity that is at most a certain constant (depending on $z_1,\ldots,z_k$) multiple of $Y$, while $\Omega_{z_1,\ldots,z_k}(X)$ denotes a quantity that is at least a certain positive constant (depending on $z_1,\ldots,z_k$) multiple of $X$. Thus, for example, the meaning of the notation $X\le O(Y)$ is identical to the meaning of the notation $X\ll Y$.

Given a subset $X$ of a group, we write $\langle X\rangle$ for the subgroup generated by $X$, although if $X$ is given explicitly as $\{x_1,\ldots,x_r\}$ then we write $\langle x_1,\ldots,x_r\rangle$ rather than $\langle\{x_1,\ldots,x_r\}\rangle$. In particular, if $x_1,\ldots,x_r$ are elements of a real vector space or Lie algebra then $\langle x_1,\ldots,x_r\rangle$ means the span of the $x_i$ over $\Z$.

\section{Doubling of ordered progressions in upper-triangular form}\label{sec:ordered.v.nil}
The main purpose of this short section is to study the doubling of ordered progressions in upper-triangular form. In particular, we show in Corollary \ref{cor:prog.doub}, below, that an $m$-proper ordered progression of rank $d$ in $C$-upper-triangular form has doubling at most $O_{C,d,m}(1)$.

Given a progression $P=P_\ord(u_1,\ldots,u_d;L)$ in upper-triangular form, for every pair $i,j$ with $i<j$ and every one of the four possible choices of sign there is by definition some (not necessarily unique) expression $u_{j+1}^{\ell_{j+1}}\cdots u_d^{\ell_d}$ for $[u_i^{\pm1},u_j^{\pm1}]$. For every pair $i,j$ with $i<j$ and every one of the four possible choices of sign we fix arbitrarily one such expression, which we call the \emph{$P$-expression} for $[u_i^{\pm1},u_j^{\pm1}]$. We then define \emph{weights} $\zeta(k)$ of the $u_k$ by setting $\zeta(k)=1$ if $u_k$ does not appear in the $P$-expression for any $[u_i^{\pm1},u_j^{\pm1}]$, and
\[
\zeta(k)=\max\{\zeta(i)+\zeta(j):\text{$u_k$ appears in the $P$-expression for some $[u_i^{\pm1},u_j^{\pm1}]$}\}
\]
otherwise. Note that this is recursively well-defined, although the definition may depend on the choice of $P$-expression.

The main result of this section is then as follows.
\begin{lemma}\label{lem:lengths.powers}
Suppose that $u_1,\ldots,u_d$ are elements of a group and that $L_1,\ldots,L_r$ are positive-integer lengths such that $(u;L)$ is in $C$-upper triangular form. Then, writing $n^\zeta L=(n^{\zeta(1)}L_1,\ldots,n^{\zeta(d)}L_d)$, we have
\[
P_\ord(u;L)^n\subset P_\ord(u;O_{C,d}(n^\zeta L)).
\]
\end{lemma}
\begin{corollary}[bounded doubling of proper ordered progressions in upper-triangular form]\label{cor:prog.doub}
If $P$ is an $m$-proper ordered progression of rank $d$ in $C$-upper-triangular form and $n\in\N$ then $|P^n|\le O_{C,m,d}(n^{O_d(1)})|P|$.
\end{corollary}

We also note that if $u_1,\ldots,u_d$ are elements of a group and $L_1,\ldots,L_d$ are positive-integer lengths then we trivially have
\begin{equation}\label{eq:lengths.powers}
P_\ord(u_1,\ldots,u_d;nL)\subset P_\ord(u_1,\ldots,u_d;L)^{dn}
\end{equation}
for every $n\in\N$ and
\begin{equation}\label{eq:prog.inverse}
P_\ord(u;L)^{-1}\subset P_\ord(u;L)^d.
\end{equation}

\begin{proof}[Proof of Lemma \ref{lem:lengths.powers}]
It is straightforward to show that we may write an arbitrary element of $\langle u_1,\ldots,u_d\rangle$ in the form
\begin{equation}\label{eq:collected.form}
u_1^{\ell_1}\cdots u_d^{\ell_d}.
\end{equation}
Indeed, in light of the trivial identity $vu=uv[v,u]$, the upper-triangular form implies that whenever $1\le i<j\le d$ and $\eps_i,\eps_j\in\{\pm1\}$ we have identities of the form
\begin{equation}\label{eq:collect.identity}
u_j^{\eps_j}u_i^{\eps_i}=u_i^{\eps_i}u_j^{\eps_j}u_{j+1}^{q_{j+1}}\cdots u_d^{q_d}
\end{equation}
with
\begin{equation}\label{eq:collect.bound}
|q_k|\le\frac{CL_k}{L_iL_j}.
\end{equation}
We first use the $i=1$ versions of these identities to write an arbitrary element of $\langle u_1,\ldots,u_d\rangle$ in the form $u_1^{\ell_1}\omega$ with $\omega\in\langle u_2,\ldots,u_d\rangle$. Applying the same argument to $\omega$ with $i=2$, and so on, we arrive at the form \eqref{eq:collected.form}.

We prove by induction on $k$ that if we start with an element $p\in P_\ord(u;L)^n$ then this process results in $|\ell_k|\le O_{C,d}(n^{\zeta(k)})L_k$. The identity \eqref{eq:collect.identity} does not result in any new copies of $u_1^{\pm1}$ compared to those featuring in the original word $p$, and so we certainly have $|\ell_1|\le nL_1\le n^{\zeta(1)}L_1$. For the inductive step, note that the only way in which new copies of $u_k$ can arise is in applying the identity \eqref{eq:collect.identity} with $i,j\le k$. However, there are at most $O_d(1)$ possible pairs $i,j\le k$ and, by induction, for any such pair the numbers of instances of the elements $u_i^{\pm1}$ and $u_j^{\pm1}$ to which we will apply the identity \eqref{eq:collect.identity} are at most $O_{C,d}(n^{\zeta(i)})L_i$ and $O_{C,d}(n^{\zeta(j)})L_j$, respectively. The number of pairs of such elements is therefore at most $O_{C,d}(n^{\zeta(i)+\zeta(j)})L_jL_j\le O_{C,d}(n^{\zeta(k)})L_jL_j$. The inductive step, and hence the lemma, therefore follows from \eqref{eq:collect.bound}.
\end{proof}

\section{Basic commutators and nilpotent progressions}\label{sec:nilp-progs}
In this section we define the term \emph{nilpotent progression} appearing in the statements of some of the theorems in the introduction, and show that a nilpotent progression is in fact a special case of an ordered progression in upper-triangular form.

We follow a set up in \cite[\S1]{bg} that was in turn based on \cite[\S11.1]{hall}. We define \emph{(formal) commutators} in the letters $x_1,\ldots,x_r$ recursively by defining each $x_i$ and $x_i^{-1}$ to be a formal commutator, and for every pair $\alpha,\alpha'$ of commutators defining $[\alpha,\alpha']$ also to be a formal commutator. We also write $[\alpha',\alpha]=[\alpha,\alpha']^{-1}$. To each commutator $\alpha$ we assign a \emph{weight vector} $\chi(\alpha)=(\chi_1(\alpha),\ldots,\chi_r(\alpha))$, defined recursively by setting $\chi_i(x_j^{\pm1})=\delta_{ij}$ and, given two formal commutators $\alpha,\alpha'$ in the $x_j$, defining $\chi([\alpha,\alpha'])=\chi(\alpha)+\chi(\alpha')$. We define the \emph{total weight} $|\chi(\alpha)|$ of a commutator $\alpha$ to be $\|\chi(\alpha)\|_1$. We call $\chi_i(\alpha)$ the \emph{weight} of $x_i$ in $\alpha$, or the \emph{$x_i$-weight} of $\alpha$. We define a commutator $[\alpha,\alpha']$ to be a \emph{trivial commutator} if $\alpha=\alpha'$ or if either $\alpha$ or $\alpha'$ is trivial.

Of course, if the letters $x_i$ are elements that generate a group $G$ then we may interpret commutators recursively via $[\alpha,\beta]=\alpha^{-1}\beta^{-1}\alpha\beta$. It is easy to see that a trivial commutator always has the identity element as its interpretation. If $G$ is $s$-step nilpotent then those commutators of total weight greater than $s$ also have trivial interpretations in $G$.

Following \cite[\S11.1]{hall}, we distinguish certain commutators, which we denote by $u_1,u_2,\ldots$, as \emph{basic commutators}. These are so called because in a free group $F$ with free generators $x_1,\ldots,x_r$ and lower central series $F=F_1>F_2>\ldots$ the basic commutators total weight $k$ in the $x_i$ form a free basis of the free abelian group $F_k/F_{k+1}$ (see \cite[\S11.1]{hall}).

We define the basic commutators recursively. For $i=1,\ldots,r$ we set $u_i=x_i$. Then, having defined the basic commutators $u_1,\ldots,u_m$ of total weight less than $k$, we define a commutator $\alpha$ of total weight $k$ to be basic if
\begin{enumerate}
\item $\alpha=[u_i,u_j]$ for some $u_i,u_j$ with $i>j$, and
\item if $u_i=[u_s,u_t]$ then $j\ge t$.
\end{enumerate}
We then label the basic commutators of total weight $k$ as $u_{m+1},\ldots,u_{m'}$, ordered arbitrarily subject to the constraint that basic commutators with the same weight vector are consecutive. This is not the same definition as that used in \cite{nilp.frei}, but the two definitions are equivalent \cite[\S11.1]{hall}. Note that the arbitrariness of the order implies that the list of basic commutators is not uniquely defined. Note, however, that if $r\ge2$ the commutators $[[\cdots[[x_2,x_1],x_1]\cdots],x_1]$ are always basic, so there are always basic commutators of every total weight, whereas if $r=1$ then $x_1$ is the unique basic commutator.

Writing $u_1,\ldots,u_d$ for the list of basic commutators of total weight at most $s$, an arbitrary element $g$ of an $s$-step nilpotent group $G$ generated by the $x_i$ can be expressed in the form
\begin{equation}\label{eq:collected}
g=u_1^{\ell_1}u_2^{\ell_2}\cdots u_d^{\ell_d}
\end{equation}
with $\ell_i\in\Z$. Indeed, we have the following result.
\begin{theorem}[{\cite[Theorem 11.2.4]{hall}}]\label{thm:collected.unique}
If $G$ is the free $s$-step nilpotent group on $x_1,\ldots,x_r$ then every element $g\in G$ has a unique representation in the form \eqref{eq:collected}.
\end{theorem}

The following definition is due to Breuillard and Green \cite{bg}.
\begin{definition}[nilpotent progression {\cite[Definition 1.4]{bg}}]
A \emph{nilpotent progression} on generators $x_1,\ldots,x_r$ in an $s$-step nilpotent group with lengths $L_1,\ldots,L_r$ is an ordered progression $P=P_\ord(u_1,\ldots,u_d;L_1,\ldots,L_d)$ on the complete list $u_1,\ldots,u_d$ of basic commutators in the $x_i$ in which
\begin{equation}\label{eq:n.prog.lengths}
L_i=L^{\chi(u_i)}
\end{equation}
for every $i>r$. Here we use the notation $L^\chi$ to denote the quantity $L_1^{\chi_1}\cdots L_r^{\chi_r}$. We define $r$ to be the \emph{rank} and $s$ to be the \emph{step} of $P$. We define the \emph{total rank} $d$ of $P$ to be the rank of $P$ as an ordered progression in the $u_i$. If $x_1,\ldots,x_r$ freely generate a free $s$-step nilpotent group of rank $r$ then we say that $P$ if a \emph{free nilpotent progression} of rank $r$ and step $s$.

If $P$ is an ordered progression in a group $G$ and $H$ is a subgroup of $G$ normalised by $P$, and if $P$ is a nilpotent progression of rank $r$, step $s$ and total rank $d$ modulo $H$, then we say $HP$ is a \emph{nilpotent coset progression} of rank $r$, step $s$ and total rank $d$.
\end{definition}
Note that the rank $r$ of a nilpotent progression $P$ is at most its total rank $d$, and indeed unless $P$ generates an abelian group the presence of at least one non-trivial basic commutator ensures that $r<d$. Nonetheless, we do have $d\ll_{r,s}1$, since $d$ is the number of basic commutators of weight at most $s$ in $r$ letters.

\begin{remark}\label{rem:free.image}
A nilpotent progression $P$ is the image of a free nilpotent progression of the same rank and step under the homomorphism taking the generators of the free group to the generators of $P$.
\end{remark}

A useful fact about nilpotent progressions is that they are already in upper triangular form, as follows.

\begin{prop}\label{prop:nilp.prog.upper.tri}
A nilpotent progression of rank $r$ and step $s$ is in $O_{r,s}(1)$-upper-triangular form.
\end{prop}

In fact, it will be useful in later sections to have a slightly more precise variant of Proposition \ref{prop:nilp.prog.upper.tri}. We define a partial order on the possible weight vectors by writing $\chi\ge\chi'$ if $\chi_i\ge\chi_i'$ for every $i$. 
\begin{prop}\label{prop:non-basic.collected}
Let $G$ be the free $s$-step nilpotent group on the generators $x_1,\ldots,x_r$, and let $\alpha$ be some (not necessarily basic) commutator in the $x_i$. Then for every basic commutator $u_i$ appearing in the expression \eqref{eq:collected} for $\alpha$ we have $\chi(u_i)\ge\chi(\alpha)$.
\end{prop}

\begin{proof}[Proof of Proposition \ref{prop:nilp.prog.upper.tri} from Proposition \ref{prop:non-basic.collected}]
If $u_i,u_j$ are two basic commutators then it follows from Proposition \ref{prop:non-basic.collected} that
\[
[u_i^{\pm1},u_j^{\pm1}]\in P_\ord(u_{k_1},\ldots,u_{k_m};O_{r,s}(1)),
\]
where $u_{k_1},\ldots,u_{k_m}$ is the ordered list of those basic commutators whose weight vectors are coordinate wise at least $\chi(u_i)+\chi(u_j)$. It follows from \eqref{eq:n.prog.lengths} that $L_{k_\ell}\ge L_iL_j$ for all such $u_{k_\ell}$, and the proposition follows.
\end{proof}

We start our proof of Proposition \ref{prop:non-basic.collected} with the case of a commutator that has $x_i$-weight at most $1$ for all $i$, which is to say a commutator in which each letter $x_i$ appears at most once (we prove this case of Proposition \ref{prop:non-basic.collected} in Lemma \ref{lem:weights.<=1}, below).

\begin{lemma}\label{lem:reduced.list}
Let $x_1,\ldots,x_r$ be letters, and let $u_1,u_2,\ldots$ be a complete list of basic commutators in the $x_i$. Then there exists a complete list of basic commutators in the letters $x_1,\ldots,x_{i-1},x_{i+1},\ldots,x_r$ that is precisely the subsequence of those $u_j$ with zero $x_i$-weight.
\end{lemma}
\begin{proof}
Define the ordered list of basic commutators of weight $1$ to be $x_1,\ldots,x_{i-1},x_{i+1},\ldots,x_r$. When defining the basic commutators of weight $n$ we may then assume by induction that the sequence of basic commutators of weight less than $n$ is precisely the subsequence of those $u_j$ of total weight less than $n$ and zero $x_i$-weight. It is then trivial that a commutator of weight exactly $n$ with zero $x_i$-weight satisfies the conditions for being included as a basic commutator in one list if and only if it satisfies the conditions for inclusion on the other list. If we then choose the order of the basic commutators of total weight $n$ in $x_1,\ldots,x_{i-1},x_{i+1},\ldots,x_r$ to be the restriction of the order on the $u_j$, it follows that the sequence of basic commutators of weight at most $n$ in $x_1,\ldots,x_{i-1},x_{i+1},\ldots,x_r$ is precisely the subsequence of those $u_j$ of total weight at most $n$ and zero $x_i$-weight, as required.
\end{proof}

\begin{lemma}\label{lem:weights.<=1}
Let $G$ be the free $s$-step nilpotent group on the generators $x_1,\ldots,x_r$, and let $\alpha$ be a commutator in the $x_i$. Then every $x_i$ that has non-zero weight in $\alpha$ also has non-zero weight in every basic commutator appearing in the expression \eqref{eq:collected} for $\alpha$.
\end{lemma}
\begin{proof}
Let $G'$ be the subgroup of $G$ generated by $x_1,\ldots,x_{i-1},x_{i+1},\ldots,x_r$, noting that $G'$ is the free $s$-step nilpotent group on these generators. Let $\pi:G\to G'$ be the unique homomorphism $G\to G'$ such that $\pi(x_i)=1$ and $\pi(x_j)=x_j$ otherwise. Expressing $\alpha$ in the form \eqref{eq:collected} as $\alpha=u_1^{\ell_1}u_2^{\ell_2}\cdots u_d^{\ell_d}$ we have
\[
\pi(\alpha)=\prod_{j\,:\,\chi_i(u_j)=0}u_j^{\ell_j}.
\]
However, if $x_i$ has non-zero weight in $\alpha$ then $\pi(\alpha)=1$, and so it follows from Theorem \ref{thm:collected.unique} and Lemma \ref{lem:reduced.list} that $\ell_j=0$ whenever $\chi_i(u_j)=0$.
\end{proof}

We now move onto the general case of Proposition \ref{prop:non-basic.collected}.

\begin{lemma}\label{lem:relabel}
Let $x_1,\ldots,x_r$ be letters, and let $u_1,u_2,\ldots$ be a complete list of basic commutators in the $x_\ell$. Let $0\le k_1\le k_2\le\ldots\le k_r$ be integers, and let $y_1,\ldots,y_{k_r}$ be letters. Then there exists a complete list $v_1,v_2,\ldots$ of basic commutators in the $y_\ell$ such that if $\rho$ is the map from commutators in the $y_\ell$ to commutators in the $x_\ell$ defined by relabelling $y_i$ as $x_j$ for $k_{j-1}<i\le k_j$, then there is a map $\xi:\N\to\N\cup\{0\}$ such that
\begin{enumerate}
\renewcommand{\labelenumi}{(\alph{enumi})}
\item$\rho(v_i)$ is a trivial commutator if $\xi(i)=0$;
\item$\rho(v_i)=u_{\xi(i)}$ for every $i$ with $\xi(i)\ne0$; and
\item if $i<j$ and $\xi(j)\ne0$ then $\xi(i)\le\xi(j)$.
\end{enumerate}
\end{lemma}
\begin{proof}
Conditions (a)--(c) hold automatically for the weight-$1$ basic commutators if we take these to be the $y_\ell$ in order, so by induction we may assume that all basic commutators $v_1,\ldots,v_m$ in the $y_\ell$ of weight less than $n$ have been chosen so that conditions (a)--(c) all hold. Suppose that $v_k=[v_i,v_j]$ is a basic commutator of weight $n$ in the $y_\ell$. We claim that $\rho(v_k)$ is either a trivial commutator or a basic commutator in the $x_\ell$. This is sufficient to prove the lemma, since we may then order those basic commutators $v$ of weight $n$ in the $y_\ell$ with $\rho(v)$ not trivial precisely so that condition (c) holds.

To prove the claim, note first that if either $\rho(v_i)$ or $\rho(v_j)$ is trivial then so is $\rho(v_k)$, so by condition (b) we may assume that
\begin{equation}\label{eq:rho}
\rho(v_k)=[u_{\xi(i)},u_{\xi(j)}].
\end{equation}
Since $v_k$ is basic in the $y_\ell$ we have $i<j$, and so by condition (c) we have either $\xi(i)=\xi(j)$ or $\xi(i)>\xi(j)$. If $\xi(i)=\xi(j)$ then \eqref{eq:rho} implies that $\rho(v_k)$ is trivial, and the claim holds. If $\xi(i)>\xi(j)$ we consider separately the cases in which $v_i$ has total weight $1$ and in which $v_i$ has total weight greater than $1$. If $v_i$ has total weight $1$ then $u_{\xi(i)}$ also has total weight $1$, and so \eqref{eq:rho} implies that $\rho(v_k)$ is basic and the claim holds. If $v_i$ has total weight greater than $1$ then we may write $v_i=[v_s,v_t]$, with $j\ge t$ since $v_k$ is basic. If either $\rho(v_s)$ or $\rho(v_t)$ is trivial then $\rho(v_k)$ is trivial and the claim holds. If not then condition (b) implies that $u_{\xi(i)}=[u_{\xi(s)},u_{\xi(t)}]$, while condition (c) implies that $\xi(j)\ge\xi(t)$, and so \eqref{eq:rho} implies that $\rho(v_k)$ is basic. This proves the claim, and hence the lemma.
\end{proof}

\begin{proof}[Proof of Proposition \ref{prop:non-basic.collected}]
For each $i=1,\ldots,r$ set $k_i=\sum_{j=1}^i\chi_j(\alpha)$, so that $|\chi(\alpha)|=k_r$. There are $k_r$ letters appearing in the commutator expression for $\alpha$. Let $\alpha'$ be the commutator obtained from $\alpha$ by relabelling these letters as $y_1,\ldots,y_{k_r}$ in turn, starting by labelling the $k_1$ copies of $x_1$ as $y_1,\ldots,y_{k_1}$, respectively, then relabelling  the $k_2-k_1$ copies of $x_2$ as $y_{k_1+1},\ldots,y_{k_2}$, respectively, and continuing in this fashion until we have relabelled the $k_r-k_{r-1}$ copies of $x_r$ as $y_{k_{r-1}+1},\ldots,y_{k_r}$, respectively. Note that the weight of each $y_i$ in $\alpha'$ is precisely $1$.

Writing $v_1,\ldots,v_n$ as the complete list of commutators of weight at most $s$ in the $y_i$ given by Lemma \ref{lem:relabel}, Theorem \ref{thm:collected.unique} implies that in the free $s$-step nilpotent group generated by the $y_i$ we have
\[
\alpha'=v_1^{m_1}\cdots v_n^{m_n}
\]
for some integers $m_i$. Lemma \ref{lem:weights.<=1} then implies that each $y_i$ has weight at least $1$ in each $v_j$ for which $m_j\ne0$.

Defining $\rho$ as in Lemma \ref{lem:relabel} we have $\rho(\alpha')=\alpha$, and hence
\[
\alpha=\prod_{i\,:\,\xi(i)\ne0}u_{\xi(i)}^{m_i}
\]
in $G$. However, it follows from condition (c) of Lemma \ref{lem:relabel} that this expression is of the form \eqref{eq:collected}, and so the proposition is proved.
\end{proof}

We close this section with an application of \cref{prop:non-basic.collected} that will be useful later. We noted in Remark \ref{rem:free.image} that a nilpotent progression $P=P_\ord(u_1,\ldots,u_d;L)$ is the image of a free nilpotent progression $\hat P=P(v;L)$ under the homomorphism mapping $v_i$ to $u_i$ for every $i$. It also follows from Proposition \ref{prop:nilp.prog.upper.tri} that $\hat P$ is in upper-triangular form. This implies that $[v_i^{\pm1},v_j^{\pm1}]$ has a $\hat P$-expression as defined at the beginning of Section \ref{sec:ordered.v.nil}, and Theorem \ref{thm:collected.unique} implies that this $\hat P$-expression is unique. We may therefore choose as a $P$-expression for $[u_i^{\pm1},u_j^{\pm1}]$ the expression obtained by relabelling each $v_k$ as $u_k$ in the $\hat P$-expression for $[v_i^{\pm1},v_j^{\pm1}]$. This expression is uniquely defined, and we call it the \emph{free $P$-expression} for $[u_i^{\pm1},u_j^{\pm1}]$

\begin{lemma}\label{lem:weights.agree}
Let $P=P_\ord(u_1,\ldots,u_d;L)$ be a nilpotent progression, and define weights $\zeta(i)$ as at the beginning of Section \ref{sec:ordered.v.nil} using the free $P$-expression for each commutator $[u_i^{\pm1},u_j^{\pm1}]$. Then for each $k$ we have $\zeta(k)=|\chi(k)|$.
\end{lemma}
\begin{proof}
We proceed by induction on $|\chi(k)|$, noting that if $|\chi(k)|=1$ then Proposition \ref{prop:non-basic.collected} implies that $\zeta(k)=1$. If $|\chi(k)|>1$ then by definition $u_k=[u_i,u_j]$ for some $i,j<k$, and so
\begin{align*}
\zeta(k)&\ge\zeta(i)+\zeta(j)&\text{(by definition of $\zeta$)}\\
    &=|\chi(i)|+|\chi(j)|&\text{(by induction)}\\
    &=|\chi(k)|&\text{(by definition of $\chi$).}
\end{align*}
On the other hand, for any $i',j'$ such that $u_k$ appears in the expression $[u_{i'}^{\pm1},u_{j'}^{\pm1}]$ we have
\begin{align*}
|\chi(k)|&\ge|\chi(i')|+|\chi(j')|&\text{(by Proposition \ref{prop:non-basic.collected})}\\
    &=\zeta(i')+\zeta(j')&\text{(by induction),}
\end{align*}
and so $|\chi(k)|\ge\zeta(k)$ by definition of $\zeta$.
\end{proof}

\section{Progressions and boxes in Lie algebras}\label{sec:coords}

We noted in Remark \ref{rem:free.image} that a nilpotent progression of rank $r$ and step $s$ is always the homomorphic image of a free nilpotent progression of rank $r$ and step $s$, which is by definition a subset of a free nilpotent group of rank $r$ and step $s$. As is well known, this free nilpotent group can in turn be embedded in a connected, simply connected nilpotent Lie group of rank $r$ and step $s$ \cite[Theorem 2.18]{rag}. We call this the \emph{free nilpotent Lie group} of rank $r$ and step $s$. It turns out that this gives us access to a fairly rich theory of additive combinatorics in nilpotent Lie groups. Such an approach has previously been exploited by Breuillard and Green \cite{bg} to prove a version of Theorem \ref{thm:nilp.Frei} for torsion-free nilpotent groups.

The central idea in the early theory of additive combinatorics in nilpotent Lie groups was to transfer everything to the Lie algebra and then apply the theory of abelian additive combinatorics. This idea was developed by Fisher, Katz and Peng \cite{fkp}, and then taken further in the Breuillard--Green paper \cite{bg}. It also played an implicit role in some earlier arguments of Tao in the Heisenberg group  \cite[Theorem 7.12]{tao.product.set}, which inspired the more general work of Fisher--Katz--Peng.

The main reason this approach is useful in the present paper is that the Lie algebra turns out to be a very convenient location in which to model the abelian geometry-of-numbers arguments we described in the introduction, as will become clear in Section \ref{sec:geom-of-nos}. However, we first need to develop some basic techniques for passing back and forth between a nilpotent Lie group and its Lie algebra, and that is the purpose of the present section.

It is well known that if $G$ is a simply connected nilpotent Lie group with Lie algebra $\g$ then there are mutually inverse diffeomorphisms $\exp:\g\to G$ and $\log:G\to\g$ \cite{bourbaki}. One can describe the group operation in $G$ in terms of addition and the Lie bracket in $\g$ via the \emph{Baker--Campbell--Hausdorff formula}, which states that for elements $X,Y\in\g$ we have
\begin{equation}\label{eq:bch}
\textstyle\exp(X)\exp(Y)=\exp(X+Y+\frac{1}{2}[X,Y]+\frac{1}{12}[X,[X,Y]]+\cdots).
\end{equation}
The precise values of the rationals appearing later in the series \eqref{eq:bch} are not important for our arguments; all that matters is that in a nilpotent Lie group the series is finite and depends only on the nilpotency class of the group.

We start in a fairly general setting. If $e_1,\ldots,e_d$ is a basis for a real vector space $V$ and $L_1,\ldots,L_d$ are non-negative integers then given a subring $A\subset \R$ we define the \emph{box} $B_A(e;L)=B_A(e_1,\ldots,e_d;L_1,\ldots,L_d)$ via
\[
B_A(e;L)=\{\ell_1e_1+\cdots+\ell_de_d:\ell_i\in A ,|\ell_i|\le L_i\}.
\]
We will be interested in the cases where $e_1,\ldots,e_d$ is a basis of a Lie algebra (with the Lie algebra is viewed as a real vector space), and $A=\Z, \Q$ or $\R$.

If $V$ is the Lie algebra of a simply connected nilpotent Lie group then, given $C>0$, we say that $(e;L)=(e_1,\ldots,e_d;L_1,\ldots,L_d)$ is in \emph{$C$-upper-triangular form} if whenever $i<j$ we have
\[
[e_i,e_j]\in B_\Z\left(e_{j+1},\ldots,e_d;\textstyle{\frac{CL_{j+1}}{L_iL_j},\ldots,\frac{CL_d}{L_iL_j}}\right).
\]
If $\pi$ is a homomorphism from $\langle\exp B_\Z(e;L)\rangle$ to some other group then $B_\Z(e;L)$ is said to be \emph{$m$-proper with respect to $\pi$} if the elements $\pi(\exp(\ell_1e_1+\cdots+\ell_de_d))$ are all distinct as the $\ell_i$ range over those integers with $|\ell_i|\le mL_i$.

Note that if a tuple $(e_1,\ldots,e_d;L_1,\ldots,L_d)$ is in $C$-upper-triangular form for some $C>0$ then the Lie algebra generated by $e_1,\ldots,e_d$ is nilpotent of step at most $d$, meaning that the set of terms with non-zero coefficients in the Baker--Campbell--Hausdorff formula \eqref{eq:bch} is a finite set depending only on $d$.

The main result of this section is as follows.

\begin{prop}\label{prop:dilates}
Let $e_1,\ldots,e_d$ be a basis of the Lie algebra $\g$ of a connected, simply connected nilpotent Lie group $G$. Let $L_1,\ldots,L_d$ be positive integers such that $(e;L)$ is in $C$-upper-triangular form, and suppose that $\exp\langle e_1,\ldots,e_d\rangle$ is a subgroup of $G$. Then, writing $u_i=\exp e_i$, we have
\begin{equation}\label{eq:dilates.concl.1}
P_\ord(u;L)\subset \exp B_\Z(e;O_{C,d}(L))
\end{equation}
and
\begin{equation}\label{eq:dilates.concl.2}
\exp B_\Z(e;L)\subset P_\ord(u;O_{C,d}(L)).
\end{equation}
Moreover, $P_\ord(u;L)$ is in $O_{C,d}(1)$-upper-triangular form. Finally, there exists a function $p_{C,d}:(0,\infty)\to(0,\infty)$ such that if $P_\ord(u;L)$ is $p_{C,d}(m)$-proper with respect to some homomorphism $\pi:\langle P_\ord(u;L)\rangle\to N$ then $B_\Z(e;L)$ is $m$-proper with respect to $\pi$, and if $B_\Z(e;L)$ is $p_{C,d}(m)$-proper with respect to $\pi$ then $P_\ord(u;L)$ is $m$-proper with respect to $\pi$.
\end{prop}

Before we present the main content of the proof, let us mention some standard theory of nilpotent Lie groups that plays an important role in our arguments. It is well known that if $G$ is a connected, simply connected nilpotent Lie group with Lie algebra $\g$ then one can use certain bases of $\g$ to define certain coordinate systems on $G$. For our purposes we record the following.
\begin{lemma}\label{lem:malcev.basis}
Let $G$ be a connected, simply connected nilpotent Lie group with Lie algebra $\g$, and let $e_1,\ldots,e_d$ be a basis for $\g$ such that whenever $i<j$ we have
\begin{equation}\label{eq:malcev.basis}
[e_i,e_j]\in\Span_\R(e_j,\ldots,e_d).
\end{equation}
Then $\exp:\g\to G$ is a bijection and, writing $u_i=\exp e_i$ for $i=1,\ldots,d$, every element of $G$ has a unique expression of the form $u_1^{\ell_1}\cdots u_d^{\ell_d}$ with $\ell_i\in\R$.
\end{lemma}
\begin{proof}
The assumption \eqref{eq:malcev.basis} says that $e_1,\ldots,e_d$ is a \emph{strong Mal'cev basis} for $\g$ in the sense of \cite[\S1.1.13]{cor-gre}, and so the lemma follows from \cite[Theorem 1.2.1 (a)]{cor-gre} and \cite[Proposition 1.2.7 (c)]{cor-gre}.
\end{proof}

Let us also say a few words about the assumption in Proposition \ref{prop:dilates} that $\exp\langle e_1,\ldots,e_d\rangle$ is a group, which at first glance might appear to be somewhat restrictive. The key reason why it is not is the following lemma, which shows that if $\Lambda=\langle e_1,\ldots,e_d\rangle$ satisfies $[\Lambda,\Lambda]\subset\Lambda$ (which is true in particular if $(e;L)$ is in upper-triangular form) then this assumption holds, at least up to finite index in some sense.
\begin{lemma}\label{lem:dilates.subgroup}
Let $G$ be a connected, simply connected nilpotent Lie group of step at most $s$ with Lie algebra $\g$, and suppose that $\Lambda$ is an additive subgroup of $\g$ with $[\Lambda,\Lambda]\subset\Lambda$. Then there exists $Q=Q_s\in\N$ such that $\exp(Q\cdot\Lambda)$ is a subgroup of $G$.
\end{lemma}
\begin{proof}
We adapt an argument that appears throughout the paper \cite{bg}. Taking $Q_s$ to be the lowest common multiple of the denominators of the rationals appearing in those terms with weight at most $s$ in the Baker--Campbell--Hausdorff formula \eqref{eq:bch}, the lemma follows from that formula.
\end{proof}

The following proposition, which we prove after Lemma \ref{lem:up-closed} below, then shows that if $P_\ord(u;L)$ is a free nilpotent progression viewed as a subset of the corresponding free nilpotent Lie group then, writing $e_i=\log u_i$, and again passing to finite index in some sense, the tuple $(e;L)$ is in upper-triangular form, and hence, in particular, generates a lattice to which Lemma \ref{lem:dilates.subgroup} applies.

\begin{prop}\label{prop:nilbox.upper-tri}
Let $u_1,\ldots,u_d$ be a complete ordered list of basic commutators in the free $s$-step nilpotent group $N_{r,s}$ on $r$ generators, viewed as a subset of the corresponding free nilpotent Lie group, and write $e_i=\log u_i$ for each $i$. Let $L_1,\ldots,L_r$ be positive integers, and for $i=r+1,\ldots,d$ write $L_i=L^{\chi(u_i)}$. Then there exist integers $Q_1,\ldots,Q_d\ll_{r,s}1$ such that $(Q_1e_1,\ldots,Q_de_d;L)$ is in $O_{r,s}(1)$-upper-triangular form.
\end{prop}

Finally, the following lemma shows how to deal with the caveat `up to finite index' attached to the previous two results.

\begin{lemma}\label{lem:coset.reps}
Let $u_1,\ldots,u_d$ be elements of a group and let $L_1,\ldots,L_d$ be positive integers such that $(u;L)$ is in $C$-upper-triangular form. Suppose moreover that every element of $\langle u_1,\ldots,u_d\rangle$ has a unique expression of the form $u_1^{\ell_1}\cdots u_d^{\ell_d}$ with $\ell_i\in\Z$. Let $Q_1,\ldots,Q_d$ be positive integers. Then
\[
P_\ord(u;L)\subset P_\ord(u;Q)\cdot P_\ord(u_1^{Q_1},\ldots,u_d^{Q_d};O_{C,d,Q}(L)).
\]
\end{lemma}
\begin{remark*}
Lemma \ref{lem:malcev.basis} implies that Lemma \ref{lem:coset.reps} applies in the setting of Proposition \ref{prop:dilates} (once we have proved that proposition).
\end{remark*}

The utility of Lemma \ref{lem:coset.reps} of course lies in the fact that $|P_\ord(u;Q)|\ll_{d,Q}1$.

\begin{proof}[Proof of Lemma \ref{lem:coset.reps}]
Let $\ell_1,\ldots,\ell_d$ be integers satisfying $|\ell_i|\le L_i$. We may assume by induction on $d$ that the lemma holds modulo $\langle u_d\rangle$, which is to say that there exist $n_1,\ldots,n_{d-1}$ with $|n_i|\ll_{C,d,Q} L_i$ and $r_1,\ldots,r_{d-1}$ with $|r_i|\le Q_i$, as well as some $p\in\Z$ such that
\begin{equation}\label{eq:mod.u_d}
u_1^{\ell_1}\cdots u_{d-1}^{\ell_{d-1}}=u_1^{r_1}\cdots u_{d-1}^{r_{d-1}}u_1^{n_1Q_1}\cdots u_{d-1}^{n_{d-1}Q_{d-1}}u_d^p,
\end{equation}
and in particular
\begin{equation}\label{eq:fin.index}
u_1^{\ell_1}\cdots u_d^{\ell_d}=u_1^{r_1}\cdots u_{d-1}^{r_{d-1}}u_1^{n_1Q_1}\cdots u_{d-1}^{n_{d-1}Q_{d-1}}u_d^{p+\ell_d}.
\end{equation}
It follows from \eqref{eq:mod.u_d}, \eqref{eq:lengths.powers}, Lemma \ref{lem:lengths.powers}, the upper-triangular form and the uniqueness of the expression $u_d^p$ that $|p|\ll_{C,d,Q}L_d$, and so there exists $n_d\in\Z$ with $|n_d|\ll_{C,d,Q}L_d$ and $r_d\in\Z$ with $|r_d|\le Q_d$ such that $p+\ell_d=n_dQ_d+r_d$. The upper-triangular form implies that $u_d$ is central, and it then follows from \eqref{eq:fin.index} that
\[
u_1^{\ell_1}\cdots u_d^{\ell_d}=u_1^{r_1}\cdots u_d^{r_d}u_1^{n_1Q_1}\cdots u_d^{n_dQ_d},
\]
and so the lemma is proved.
\end{proof}

We now pass to the main details of the proofs of Propositions \ref{prop:dilates} and \ref{prop:nilbox.upper-tri}. Given elements $v_1,\ldots,v_r$ of a Lie algebra we define the \emph{(formal) Lie brackets} in the $v_i$ analogously to how we define formal commutators. Specifically, we define every $v_j$ to be a Lie bracket of weight $1$ in the $v_i$, and for every pair $\alpha,\alpha'$ of Lie brackets of weights $\omega,\omega'$, respectively, we define $[\alpha,\alpha']$ to be a Lie bracket in the $v_i$ of weight $\omega+\omega'$.

Following \cite[Definition 3.2]{nilp.frei}, we also define certain functions mapping a set of letters to a commutator or Lie bracket in those letters. We momentarily treat brackets as formal objects that can be interpreted either as commutators or as Lie brackets depending on the context. Given letters $v_1,\ldots,v_r$, the function $\alpha_i$ defined by $\alpha_i(v_1,\ldots,v_r)=v_i$ is a bracket form of weight $1$, and then given two bracket forms $\alpha,\alpha'$ of weights $\omega,\omega'$, respectively, the function $[\alpha,\alpha']$ defined by $[\alpha,\alpha'](v_1,\ldots,v_r)=[\alpha(v_1,\ldots,v_r),\alpha'(v_1,\ldots,v_r)]$ is a bracket form of weight $\omega+\omega'$. Thus, for example, the function $\alpha:(w_1,w_2)\mapsto[w_1,[w_1,w_2]]$ is a bracket form of weight $3$, and if $x_1,x_2$ are elements of a group then $\alpha(x_1,x_2)$ is the commutator $[x_1,[x_1,x_2]]$, whilst if $v_1,v_2$ are elements of a Lie algebra then $\alpha(v_1,v_2)$ is the Lie bracket $[v_1,[v_1,v_2]]$.
\begin{lemma}\label{lem:comms.1}
Let $\alpha$ be a bracket form of weight $m$. Then there exists a sequence $\beta_1,\beta_2,\ldots$ of bracket forms of weight greater than $m$, of which at most finitely many have any given weight, and rationals $q_1,q_2,\ldots$ such that if $x_1,\ldots,x_m$ are elements of a connected, simply connected nilpotent Lie group, and $v_i=\log x_i$ are elements of the corresponding Lie algebra, then
\[
\log\alpha(x_1,\ldots,x_m)=\alpha(v_1,\ldots,v_m)+q_1\beta_1(v_1,\ldots,v_m)+q_2\beta_2(v_1,\ldots,v_m)+\cdots,
\]
with each $\beta_j$ featuring each $v_i$ at least once.
\end{lemma}
\begin{proof}
The result is trivial for $m=1$, so by induction we may assume that the result is true for all bracket forms of weight less than $m$. However, by definition we have $\alpha=[\gamma_1,\gamma_2]$ for some forms $\gamma_1,\gamma_2$ of weight less than $m$, and so applying the Baker--Campbell--Hausdorff formula \eqref{eq:bch} to the string $\gamma_1^{-1}\gamma_2^{-1}\gamma_1\gamma_2$ yields the desired result.
\end{proof}

Let $u_1,\ldots,u_d$ be a complete ordered list of basic commutators in the free $s$-step nilpotent Lie group $G$ on generators $x_1,\ldots,x_r$, and let $e_i=\log u_i$ be elements of the corresponding Lie algebra $\g$. Define recursively the \emph{adjusted weight vector} $\omega(\alpha)$ of a formal Lie bracket $\alpha$ in the $e_i$ by setting $\omega(e_i)=\chi(u_i)$ and setting $\omega([\alpha_1,\alpha_2])=\omega(\alpha_1)+\omega(\alpha_2)$ whenever $\alpha_1$ and $\alpha_2$ are formal Lie brackets whose adjusted weight vectors have already been defined.

As in Section \ref{sec:nilp-progs}, we define a partial order on the weight vectors of commutators in the $x_i$ by declaring that $\chi\ge\chi'$ if $\chi_i\ge\chi'_i$ for every coordinate $i$. We say that an increasing sequence $u_{i_1},\ldots,u_{i_p}$ of basic commutators is \emph{upwards closed} if for every $u_{i_j}$ in the sequence and every $u_k$ with $\chi(u_k)>\chi(u_{i_j})$ we have $u_k$ also in the sequence. For each $v\in\Z^r$ we write
\[
\g_v^\Q=\Span_\Q\{e_i:\chi(u_i)\ge v\}.
\]

\begin{lemma}\label{lem:up-closed}
Let $u_1,\ldots,u_d$ be a complete ordered list of basic commutators in the free $s$-step nilpotent Lie group $G$ on generators $x_1,\ldots,x_r$, and let $e_i=\log u_i$ be elements of the corresponding Lie algebra $\g$. Then for every Lie bracket $\alpha$ in the $e_i$ we have $\alpha\in\g_{\omega(\alpha)}^\Q$.
\end{lemma}
\begin{proof}
The lemma is trivial when $r=1$, so we may assume that $r\ge2$. In that case we prove, for each $m$, that the following assertions hold.
\begin{enumerate}
\item If $\alpha$ is a Lie bracket in the $e_i$ with $|\omega(\alpha)|>|\chi(u_m)|$ then $\alpha\in\g_{\omega(\alpha)}^\Q$.
\item If $u_{i_1},\ldots,u_{i_p}$ is an upwards-closed subsequence of basic commutators with $i_1\ge m$ then for all rationals $\ell_{i_1},\ldots,\ell_{i_p}\in\Q$ we have $\log u_{i_1}^{\ell_{i_1}}\cdots u_{i_p}^{\ell_{i_p}}\in\Span_\Q(e_{i_1},\ldots,e_{i_p})$.
\end{enumerate}
This is sufficient, since if $|\omega(\alpha)|>1$ then the lemma follows from the $m=1$ case of (1), whereas if $|\omega(\alpha)|=1$ then $\alpha=e_i$ for some $i$ and the lemma is trivially satisfied.

Assertions (1) and (2) are trivially true if $m=d$, so we may fix $m$ and assume by induction that both assertions hold for all larger values of $m$.

We start with assertion (2), assuming that $i_1\ge m$ and that $\ell_{i_1},\ldots,\ell_{i_p}\in\Q$. The sequence $u_{i_2},\ldots,u_{i_p}$ is upwards closed, so the inductive hypothesis for assertion (2) implies that, writing $y=\log u_{i_2}^{\ell_{i_2}}\cdots u_{i_p}^{\ell_{i_p}}$, we have $y\in\Span_\Q(e_{i_2},\ldots,e_{i_p})$. The Baker--Campbell--Hausdorff formula \eqref{eq:bch} then implies that
\[
\textstyle\log u_{i_1}^{\ell_{i_1}}\cdots u_{i_p}^{\ell_{i_p}}=\ell_{i_1}e_{i_1}+y+\frac{\ell_{i_1}}{2}[e_{i_1},y]+\cdots,
\]
and hence that $\log u_{i_1}^{\ell_{i_1}}\cdots u_{i_p}^{\ell_{i_p}}$ is a rational linear combination of $e_{i_1},\ldots,e_{i_p}$ and some set of Lie brackets, each of which has $e_{i_1}$ and at least one $e_{i_j}$ with $j\ge2$ amongst its components. The inductive hypothesis for assertion (1) and the fact that $u_{i_1},\ldots,u_{i_p}$ is upwards closed implies that each of these Lie brackets is itself a rational linear combination of $e_{i_2},\ldots,e_{i_p}$, and so assertion (2) is proved.

We now move to assertion (1), assuming that $\alpha$ is a Lie bracket in the $e_i$ with $|\omega(\alpha)|>|\chi(u_m)|$. Writing $\alpha=\alpha(e_{j_1},\ldots,e_{j_k})$, it follows from Lemma \ref{lem:comms.1} that there exist rationals $q_1,\ldots,q_n$ and Lie brackets $\beta_1,\ldots,\beta_n$ in the $e_i$, each of which satisfies $\omega(\beta_j)>\omega(\alpha)$, such that
\[
\alpha=\log\alpha(u_{j_1},\ldots,u_{j_k})+\sum_{i=1}^n q_i\beta_i.
\]
Since there are basic commutators of every weight when $r\ge2$, the fact that $\omega(\beta_i)>\omega(\alpha)$ implies in particular that $|\omega(\beta_i)|>|\chi(u_{m+1})|$ and $\g_{\omega(\beta_i)}^\Q\subset\g_{\omega(\alpha)}^\Q$. The induction hypothesis for assertion (1) therefore implies that each $\beta_i$ satisfies $\beta_i\in\g_{\omega(\alpha)}^\Q$. We therefore have
\[
\alpha\in\log\alpha(u_{j_1},\ldots,u_{j_k})+\g_{\omega(\alpha)}^\Q,
\]
and so assertion (1) follows from Proposition \ref{prop:non-basic.collected} and assertion (2) applied to the upwards-closed set $\{u_i:\chi(u_i)\ge\omega(\alpha)\}$.
\end{proof}

\begin{proof}[Proof of Proposition \ref{prop:nilbox.upper-tri}]
It follows from Lemma \ref{lem:up-closed} that for every $i,j$ we have
\[
[e_i,e_j]\in B_\Q(e_{k_1},\ldots,e_{k_n};O_{r,s}(1)),
\]
where $u_{k_1},\ldots,u_{k_m}$ is the ordered list of those basic commutators whose weight vectors are coordinatewise at least $\chi(u_i)+\chi(u_j)$. We may therefore pick natural numbers $Q_d,\ldots,Q_1$ in turn so that
\[
[Q_ie_i,Q_je_j]\in B_\Z(Q_{k_1}e_{k_1},\ldots,Q_{k_n}e_{k_n};O_{r,s}(1)).
\]
However, $L_{k_\ell}\ge L_iL_j$ for every $\ell$ by definition, and so the proposition follows.
\end{proof}

We now move onto the proof of Proposition \ref{prop:dilates}. We start by recording the following observation as a lemma for ease of later reference.
\begin{lemma}\label{lem:upper-tri.dilate}
Let $k>0$. Then for elements $x_1,\ldots,x_d$ of a group or a Lie algebra, if the tuple $(x;L)$ is in $C$-upper-triangular form then the tuple $(x;kL)$ is in $Ck$-upper-triangular form.
\end{lemma}

\begin{lemma}\label{lem:upper-tri}
Let $e_1,\ldots,e_d$ be elements of a Lie algebra and $L_1,\ldots,L_d$ non-negative integers such that $(e;L)$ is in $C$-upper-triangular form. Let $\beta$ be a bracket form of weight $r$. Then for every $i_1\le\ldots\le i_r$ we have
\[
\textstyle\beta(e_{i_1},\ldots,e_{i_r})\in B_\Z\left(e_{i_r+1},\ldots,e_d;\textstyle{\frac{O_{C,d,r}(L_{i_r+1})}{L_{i_1}\cdots L_{i_r}},\ldots,\frac{O_{C,d,r}(L_d)}{L_{i_1}\cdots L_{i_r}}}\right)
\]
\end{lemma}
\begin{proof}
This is a routine induction on $r$.
\end{proof}

\begin{lemma}\label{lem:QL.group}
Let $e_1,\ldots,e_d$ be a basis of the Lie algebra $\g$ of a connected, simply connected nilpotent Lie group $G$, and let $L_1,\ldots,L_d$ be positive integers such that $(e;L)$ is in $C$-upper-triangular form. Then
\begin{equation}\label{eq:sm.doub.rational}
(\exp B_\Z(e;L))^2\subset\exp B_\Q(e;O_{C,d}(L)).
\end{equation}
In particular, if $\exp\langle e_1,\ldots,e_d\rangle$ is a subgroup of $G$ then
\[
(\exp B_\Z(e;L))^2\subset\exp B_\Z(e;O_{C,d}(L)).
\]
\end{lemma}
\begin{proof}
This follows from the Baker--Campbell--Hausdorff formula \eqref{eq:bch} and Lemma \ref{lem:upper-tri}.
\end{proof}

\begin{corollary}
If $e_1,\ldots,e_d$ is a basis of the Lie algebra $\g$ of a connected, simply connected nilpotent Lie group $G$, and $L_1,\ldots,L_d$ are positive integers such that $(e;L)$ is in $C$-upper-triangular form, then $\exp B_\Z(e;L)$ has doubling at most $O_{C,d}(1)$.
\end{corollary}
\begin{proof}
The rationals appearing in the right-hand side of \eqref{eq:sm.doub.rational} arise from a single application of the Baker--Campbell--Hausdorff formula, and so have denominators bounded in terms of $d$.
\end{proof}

\begin{proof}[Proof of Proposition \ref{prop:dilates}]
The inclusion \eqref{eq:dilates.concl.1} follows from repeated application of Lemmas \ref{lem:upper-tri.dilate} and \ref{lem:QL.group}. To prove \eqref{eq:dilates.concl.2}, observe using the Baker--Campbell--Hausdorff formula \eqref{eq:bch} that for $\ell_i\in\Z$ we have
\[
\exp(-\ell_1e_1)\exp(\ell_1e_1+\cdots+\ell_de_d)\subset\exp\Span_\R(e_2,\ldots,e_d)
\]
Since $e_1,\ldots,e_d$ is a basis for $\g$ and $\exp:\g\to G$ is injective, this combines with Lemma \ref{lem:QL.group} to imply that
\[
\exp B_\Z(e;L)\subset P_\ord(u_1;L_1)\exp B_\Z(e_2,\ldots,e_d;O_{C,d}(L_2),\ldots,O_{C,d}(L_d)),
\]
from which \eqref{eq:dilates.concl.2} follows by induction and Lemma \ref{lem:upper-tri.dilate}.

To see that $P_\ord(u;L)$ is in $O_{C,d}(1)$-upper-triangular form, note first that Lemma \ref{lem:comms.1} followed by Lemma \ref{lem:upper-tri} imply that for $i<j$ we have
\[
[u_i^{\pm1},u_j^{\pm1}]\in\exp B_\Q\left(e_{j+1},\ldots,e_d;\textstyle{\frac{O_{C,d}(L_{j+1})}{L_iL_j},\ldots,\frac{O_{C,d}(L_d)}{L_iL_j}}\right),
\]
and hence
\[
[u_i^{\pm1},u_j^{\pm1}]\in\exp B_\Z\left(e_{j+1},\ldots,e_d;\textstyle{\frac{O_{C,d}(L_{j+1})}{L_iL_j},\ldots,\frac{O_{C,d}(L_d)}{L_iL_j}}\right)
\]
since $\exp\langle e_1,\ldots,e_d\rangle$ is a group. It therefore follows from Lemma \ref{lem:upper-tri.dilate} and \eqref{eq:dilates.concl.2} applied to
\[
B_\Z\left(e_{j+1},\ldots,e_d;\textstyle{\frac{O_{C,d}(L_{j+1})}{L_iL_j},\ldots,\frac{O_{C,d}(L_d)}{L_iL_j}}\right)
\]
that
\[
[u_i^{\pm1},u_j^{\pm1}]\in P_\ord\left(u_{j+1},\ldots,u_d;\textstyle{\frac{O_{C,d}(L_{j+1})}{L_iL_j},\ldots,\frac{O_{C,d}(L_d)}{L_iL_j}}\right).
\]

Finally, Lemma \ref{lem:upper-tri.dilate} implies that for every $m$ the tuple $(e,mL)$ is in $Cm$-upper-triangular form, and so \eqref{eq:dilates.concl.1} and \eqref{eq:dilates.concl.2} imply that there exists $p_{C,d}:(0,\infty)\to(0,\infty)$ such that
\[
P_\ord(u;mL)\subset\exp B_\Z(e;p_{C,d}(m)L)
\]
and
\[
\exp B_\Z(e;mL)\subset P_\ord(u;p_{C,d}(m)L).
\]
In light of Lemma \ref{lem:malcev.basis}, this proves the final assertion of the proposition.
\end{proof}

\section{Geometry of numbers}\label{sec:geom-of-nos}

As we described in the introduction, the geometry of numbers plays an important role in the proof of the abelian version of Theorem \ref{thm:nilp.bilu}. In this section we describe how to transfer this aspect of the argument to the nilpotent setting.

Given a set $A$ in a Lie algebra we write $[A,A]=\{[a,a']:a,a'\in A\}$. We say that a symmetric convex body in $\R^d$ is \emph{strictly thick} with respect to a lattice $\Lambda$ if there exists some $\lambda<1$ such that $\lambda B\cap\Lambda$ generates $\Lambda$. The main result of this section is the following, which may be thought of as a nilpotent version of part of the proof of \cite[Theorem 1.2]{bilu} (see in particular \cite[(3.2) \& (3.3)]{bilu}).

\begin{prop}\label{prop:nilp.box}
Let $\g$ be a nilpotent Lie algebra of dimension $d$, and let $\Lambda$ be a lattice in $\g$ satisfying $[\Lambda,\Lambda]\subset\Lambda$. Suppose that $B$ is a strictly thick symmetric convex body in $\g$ satisfying $[B,B]\subset B$. Then there exists a basis $e_1,\ldots,e_d$ for $\Lambda$ and integers $L_1,\ldots,L_d$ such that
\begin{equation}\label{eq:nilp.box.concl}
B\subset B_\R(e;L)\subset O_d(1)B,
\end{equation}
and such that $(e;L)$ is in $1$-upper-triangular form.
\end{prop}

Throughout this section and the rest of the paper, if $B$ is a symmetric convex body in $\R^d$ we denote by $\|\cdot\|_B$ the norm on $\R^d$ whose unit ball is the closure $\overline B$ of $B$.

\begin{lemma}[{\cite[Lemma 6.6]{bilu}}]\label{lem:rom.5}
Let $\|\cdot\|$ be a Euclidean norm on $\R^d$, and given $v\in\R^d$ write $\pi_v$ for the orthogonal projection of $\R^d$ onto the orthogonal complement of $\Span_\R(v)$. Suppose $B$ is a symmetric convex body and $v\in\R^d$ with $v\ne0$. Then
$$\vol\left(\pi_v(B)\right)\leq \frac{d}{2}\frac{\|v\|_B}{\|v\|}\vol(B).$$
\end{lemma}

Recall that the \emph{successive minima} $\lambda_1\le\ldots\le\lambda_d$ of a convex body $B\subset\R^d$ with respect to a lattice $\Lambda<\R^d$ are defined by $\lambda_i=\inf\{\lambda\in\R:\dim\Span_\R(\lambda B\cap\Lambda)\ge i\}$. We call a sequence $a_1,\ldots,a_d$ of elements of $\Lambda$ \emph{witnesses} to the successive minima if they are linearly independent over $\R$ and if $a_1,\ldots,a_i\subset\lambda_i\overline B$ for each $i$.

In a similar argument to \cite{bilu}, we make use of the following lemma, which is essentially \cite[Ch.~VIII, corollary of Theorem VII]{cassels}.
\begin{lemma}\label{lem:mahler}
Let $\Lambda\subset\R^d$ be a lattice, and let $B$ be a symmetric convex body in $\R^d$ with successive minima $\lambda_1,\ldots,\lambda_d$ with respect to $\Lambda$ witnessed by $a_1,\ldots,a_d$. Then there exists a basis $e_1,\ldots,e_d$ for $\Lambda$ such that
\begin{enumerate}
\item $\|e_1\|_B=\lambda_1$;
\item $\|e_i\|_B\le\frac{i}{2}\lambda_i$\qquad$(2\le i\le d)$;
\item $\Span_\R(e_1,\ldots,e_i)=\Span_\R(a_1,\ldots,a_i)$\qquad$(1\le i\le d)$.
\end{enumerate}
\end{lemma}
\begin{proof}
The lemma follows from applying \cite[Ch.~V, Lemma 8]{cassels} with $F=\|\cdot\|_B$. The third condition is not stated explicitly there, but follows automatically from the construction of the elements $e_i$, as can be seen from equations (3) and (4) in the proof.
\end{proof}
\begin{remark*}
A basis for $\Lambda$ satisfying conditions (1) and (2) of Lemma \ref{lem:mahler} is sometimes called a \emph{Mahler basis} of $B$.
\end{remark*}

\begin{proof}[Proof of Proposition \ref{prop:nilp.box}]
Let $a_1,\ldots,a_d$ be witnesses to the successive minima $\lambda_1,\ldots,\lambda_d$ of $B$, and let $e_1,\ldots,e_d$ be the basis of $\Lambda$ given by Lemma \ref{lem:mahler}. Bilu shows in the proof of \cite[Theorem 1.2]{bilu} that conditions (1) and (2) of Lemma \ref{lem:mahler} are sufficient to imply that there are positive reals $L_1,\ldots,L_d$ such that $B\subset B_\R(e;L)\subset O_d(1)B$. The fact that $B$ is strictly thick implies that $L_i\ge1$ for each $i$, so at the expense of some loss in the implied constant $O_d(1)$ we may in fact assume that $L_i\in\N$ for each $i$. The basis $e_1,\ldots,e_d$ and integers $L_i$ therefore satisfy \eqref{eq:nilp.box.concl}.

We may assume without loss of generality that $B$ is closed, and hence that $a_i\in\lambda_i B$ for each $i$. The fact that $[B,B]\subset B$ therefore implies that whenever $1\le i<j\le d$ we have $\lambda_i^{-1}\lambda_j^{-1}[a_i,a_j]\subset B$, and hence $[a_i,a_j]\subset\lambda_i\lambda_jB$. Since $B$ is strictly thick we have $\lambda_j<1$, and so by definition of the successive minima we conclude that $[a_i,a_j]\subset\Span_\R(a_1,\dots,a_{i-1})$. Condition (3) of Lemma \ref{lem:mahler} therefore implies that whenever $1\le i<j\le d$ we have $[e_i,e_j]\subset\Span_\R(e_1,\dots,e_{i-1})$. Reversing the order of the $e_i$, and since $[\Lambda,\Lambda]\subset\Lambda$ implies in particular that $[e_i,e_j]\subset\Lambda$, we conclude that whenever $1\le i<j\le d$ we have
\begin{equation}\label{eq:upper.tri}
[e_i,e_j]\in\langle e_{j+1},\ldots,e_d\rangle.
\end{equation}

To obtain the stronger condition that $(e;L)$ is in $1$-upper-triangular form, we increase some of the integers $L_i$ by factors bounded in terms of $d$ only; this clearly does not affect the truth of condition \eqref{eq:nilp.box.concl} apart from some worsening of the implied constants. Indeed, we show by induction on $k$ that it is possible to increase the $L_i$ in this way to ensure that whenever $1\le i<j\le k$ we have
\begin{equation}\label{eq:upper-tri}
[e_i,e_j]\in B_\Z\left(e_{j+1},\ldots,e_d;\textstyle{\frac{L_{j+1}}{L_iL_j},\ldots,\frac{L_d}{L_iL_j}}\right).
\end{equation}
To that end, let $1\le k\le d$ and suppose that \eqref{eq:upper-tri} holds whenever $1\le i<j<k$. 
Condition \eqref{eq:nilp.box.concl} and the assumption that $[B,B]\subset B$ imply that for each $i=1,\ldots,k-1$ we have
\begin{align*}
[L_ie_i,L_ke_k] &\subset O_d(1)[B,B]\\
         &\subset O_d(1)B\\
         &\subset B_\R(e;O_d(L)),
\end{align*}
and so \eqref{eq:upper.tri} and the fact that the $e_i$ form a basis for $\Lambda$ imply that
\[
[L_ie_i,L_ke_k]\subset B_\Z(e_{k+1},\ldots,e_d;O_d(L_{k+1}),\ldots,O_d(L_d)).
\]
Upon mutiplying each of $L_{k+1},\ldots,L_d$ by a constant depending only on $d$, we may therefore ensure that condition \eqref{eq:upper-tri} is satisfied whenever $1\le i<j\le k$.
\end{proof}

\section{Generation of progressions}\label{sec:gen.of.progs}
The main purpose of this section is to prove the following result. The proof we give is considerably simpler than in the original version of the paper; we are grateful to the anonymous referee for suggesting it.
\begin{prop}\label{prop:rom.8}
For every $d\in\N$ there exists a constant $M=M_d$ such that if $e_1,\ldots,e_d$ is the standard basis for $\Z^d$, if $L_1,\ldots,L_d$ are integers at least $M$, and if $A\subset B_\Z(e;L)$ is a symmetric generating set for $\Z^d$ satisfying $|A|\geq cL_1\ldots L_d$ for some $c>0$, then $B_\Z(e;L)\subset O_{c,d}(1)A$.
\end{prop}
Since $\Z^d$ is most naturally written as an additive group, here we write $nA=\{a_1+\ldots+a_n:a_i\in A\}$. In particular, $x\in O(1)A$ if there exists $t\ll1$ and $a_1,\ldots,a_t\in A$ such that  $x=a_1+\ldots+a_t$.

We also deduce the following more general corollary, which we do not use in this paper but which may be of independent interest.
\begin{corollary}\label{cor:rom.8}
For every $r,s\in\N$ there exists a constant $M=M_{r,s}$ such that if $P$ is a free nilpotent progression of rank $r$ and step $s$ with side lengths at least $M$, and if $A\subset P$ is a symmetric generating set for $\langle P\rangle$ such that $|A|\ge c|P|$ for some $c>0$, then $P\subset A^{O_{c,r,s}(1)}$.
\end{corollary}

\begin{remark}
Corollary \ref{cor:rom.8} does not hold for an arbitrary nilpotent progression, even in the abelian case. For example, if $P=\{-2,-1,0,1,2\}$ inside $\Z/n\Z$ for some large odd $n$ then the set $A=\{-2,0,2\}$ is a symmetric generating set for $\Z/n\Z$ with $|A|\ge\frac{1}{2}|P|$, but we do not have $P\subset O(1)A$ as $n\to\infty$. One can also take $P=\{-n-1,-n,-n+1,-1,0,1,n-1,n,n+1\}$ and $A=\{\pm1\}$ in $\Z$.
\end{remark}

We use the following standard facts from additive combinatorics. 

\begin{lemma}\label{lem:Frei.comp}
Let $G$ be a group, let $\eps<1$, and suppose that $A$ is a finite symmetric subset of $G$ satisfying $|A^3|\le(1+\eps)|A|$. Then $A^2$ is a subgroup of $G$ of size at most $(1+\eps)|A|$.
\end{lemma}
\begin{proof}
It is certainly true that $|A^2|\le(1+\eps)|A|$. Note also that if $x,y\in A^2$ then $x^{-1}A$ and $yA$ have non-trivial intersection. In particular, this implies that $xy\in A^2$, and so $A^2$ is a group.
\end{proof}
\begin{lemma}\label{lem:proportionFiniteGroup}
Let $c>0$, let $G$ be a finite group and suppose that $A$ is a symmetric generating subset of $G$ satisfying $|A|\ge c|G|$. Then $G=A^{O_c(1)}$.
\end{lemma}
\begin{proof}
This follows from repeated application of Lemma \ref{lem:Frei.comp} with $\eps=\frac{1}{2}$, say.
\end{proof}

\begin{proof}[Proof of \cref{prop:rom.8}]
Since $|B_\Z(e;L)|=\prod_{i=1}^d(2L_i+1)$, the quantities $2L_i+1$ appear naturally in many places in this proof. However, to make the notation less cumbersome we use the rather loose bounds $(2L_i+1)\le3L_i$.

For each $i=1,\ldots,d$, the pigeonhole principle implies that there exists $x\in A$ such that $|A\cap(x+\langle e_i\rangle)|\ge3^{1-d}cL_i$, and hence that $|2A\cap\langle e_i\rangle)|\ge3^{1-d}cL_i$. There therefore exist $r_1,\ldots,r_d\in\Z$ with $r_i\ge3^{-d}L_i$ for each $i$ such that $r_ie_i\in2A$ for each $i$. Let $\Lambda=\langle r_ie_i\rangle$, noting that $A$ generates $\Z^d/\Lambda$ and that $|(A+\Lambda)/\Lambda|\ge3^{-d^2}cL_1\cdots L_d\ge2^{-d}3^{-d^2}c|\Z^d/\Lambda|$. It therefore follows from \cref{lem:proportionFiniteGroup} that $\Z^d=O_{c,d}(1)A+\Lambda$.

This implies in particular that $B_\Z(e;L)\subset O_c(1)A+\Lambda$. Given $b\in B_\Z(e;L)$, therefore, we may write $b=a+z$ with $a\in O_{c,d}(1)A$ and $z\in\Lambda$. This implies in particular that $z=b-a\in O_{c,d}(1)B_\Z(e;L)$, and hence $z\in O_{c,d}(1)\{r_1e_1,\ldots,r_de_d\}\subset O_{c,d}(1)A$. Thus $b\in O_{c,d}(1)A$, as required.
\end{proof}

\begin{proof}[Proof of Corollary \ref{cor:rom.8}]
We proceed by induction on $s$, noting that when $s=1$ this is simply Proposition \ref{prop:rom.8}. By definition there exists a free $s$-step nilpotent group $G$ on generators $x_1,\ldots,x_r$ such that, writing $u_1,\ldots,u_d$ for the ordered list of basic commutators of weight at most $s$ in the $x_i$, we have $P$ of the form $P=P_\ord(u;L)$ for some $L_1,\ldots,L_d$ satisfying \eqref{eq:n.prog.lengths}.

Writing $\pi$ for the projection $\pi:G\to G/G_s$, the induction hypothesis implies that
\begin{equation}\label{eq:projected.1}
\pi(P)\subset\pi(A)^{O_{c,r,s}(1)},
\end{equation}
and in particular that
\[
P\subset A^{O_{c,r,s}(1)}(A^{O_{c,r,s}(1)}P\cap G_s)\subset A^{O_{c,r,s}(1)}(P^{O_{c,r,s}(1)}\cap G_s).
\]
However, writing $u_t,\ldots,u_d$ for the basic commutators of weight exactly $s$ and denoting
\[
B=P_\ord(u_t,\ldots,u_d;L^{\chi(u_t)},\ldots,L^{\chi(u_d)}),
\]
it follows from Lemma \ref{lem:lengths.powers} that $P^{O_{c,r,s}(1)}\subset  P_\ord(u;O_{c,r,s}(L))$, and in particular $P^{O_{c,r,s}(1)}\cap G_s\subset B^{O_{c,r,s}(1)}$, and so in fact we have
\begin{equation}\label{eq:projected.2}
P\subset A^{O_{c,r,s}(1)}\cdot B^{O_{c,r,s}(1)}.
\end{equation}
The pigeonhole principle implies that there exists $p\in\pi(P)$ such that $|\pi^{-1}(p)\cap A|\ge c|B|$, and hence that $|G_s\cap A^2|\ge c|B|$. However, another application of Lemma \ref{lem:lengths.powers} implies that $A^2\cap G_s\subset A^2\cap B^{O_{r,s}(1)}$, and so there exists $m\ll_{r,s}1$ such that
\begin{equation}\label{eq:G_s.dense}
|A^2\cap B^m|\ge c|B|\gg_{r,s}c|B^m|.
\end{equation}
A commutator of weight $s$ in the $x_i$ depends only on the images of the $x_i$ in $G/[G,G]$, and so \eqref{eq:projected.1} implies that each $u_i$ of weight $s$ is contained in $A^{O_{c,r,s}(1)}$. In particular, there is a generating set for $G_s$ contained in $A^{O_{c,r,s}(1)}\cap B$. In light of \eqref{eq:G_s.dense}, Proposition \ref{prop:rom.8} therefore implies that $B\subset A^{O_{c,r,s}(1)}$, and so the result follows from \eqref{eq:projected.2}.
\end{proof}

\section{Construction of a proper progression}\label{sec:Bilu}

The main aim of this section is to prove Theorem \ref{thm:nilp.bilu}. The starting point of the argument is to view the nilpotent progression $P_0$ as the image of a free nilpotent progression as described in \cref{rem:free.image}. We then view this free nilpotent progression as lying in a free nilpotent Lie group and apply the geometry-of-numbers arguments from \cref{sec:geom-of-nos} in the Lie algebra of this free nilpotent Lie group.

In order to keep track of the various objects featuring in this argument, it will be convenient to introduce the following definition.
\begin{definition}[Lie coset progression]\label{def:Lie.prog}
Let $N$ be a group, let $H$ be a finite subgroup of $N$, let $G$ be a connected, simply connected nilpotent Lie group with Lie algebra $\g$, let $e_1,\ldots,e_d$ be a basis of $\g$ such that $\Gamma=\exp\langle e_1,\ldots,e_d\rangle$ is a group, let $L_1,\ldots,L_d\in\N$ be such that $(e,L)$ is in $C$-upper-triangular form, let $\pi:\Gamma\to N$ be a map such that $H\lhd\langle\pi(\Gamma)H\rangle$ and such that $\pi$ is a homomorphism modulo $H$, and write $u_i=\exp(e_i)$ for each $i$. Then the set
\[
P_\Lie(N,H,G,\g,\Gamma,\pi,e,u,L)=H\pi(P_\ord(u;L))
\]
is said to be a \emph{Lie coset progression of dimension $d$ in $C$-upper-triangular form} in $N$. We say that $P_\Lie(N,H,G,\g,\Gamma,\pi,e,u,L)$ is $m$-proper if the ordered coset progression $HP_\ord(u;L)$ is $m$-proper. If $H=\{1\}$ then we say simply that $P_\Lie(N,H,G,\g,\Gamma,\pi,e,u,L)$ is a \emph{Lie progression of dimension $d$ in $C$-upper-triangular form}.
\end{definition}

The main result of this section is then the following, which implies Theorem \ref{thm:nilp.bilu} and, in conjunction with Theorems \ref{thm:nilp.Frei} and \ref{thm:bgt.noprog} and Lemmas \ref{lem:lengths.powers} and \ref{lem:malcev.basis}, yields Corollary \ref{cor:Lie.model}.

\begin{theorem}\label{thm:nilp.bilu.detailed}
Let $P_0$ be a nilpotent progression of rank $r$ and step $s$. Then for every $m,C>0$ there exists a normal subgroup $H\lhd\langle P_0\rangle$ with $H\subset P_0^{O_{C,r,s,m}(1)}$, an $m$-proper Lie coset progression $P=P_\Lie(\langle P_0\rangle,H,G,\g,\Gamma,\pi,e,u,L)$ of dimension $d\ll_{r,s}1$ in $O_{C,r,s}(1)$-upper-triangular form such that the tuple $(u;L)$ is in $C$-upper-triangular form, and a set $X\subset\langle P_0\rangle$ with $|X|\ll_{r,s}1$ such that $HP_0\subset XP\subset HP_0^{O_{C,r,s,m}(1)}$.
\end{theorem}

Using \cref{rem:free.image} and embedding the free nilpotent group inside the free nilpotent Lie group as described above already brings us fairly close to the conclusion of \cref{thm:nilp.bilu.detailed}. The main issue with this is that $P_0$ may not be proper. The following proposition allows us to obtain the desired properness.

\begin{prop}\label{prop:bilu.main}
Let $N$ be a group, and let $P=P_\Lie(N,\{1\},G,\g,\Gamma,\pi,e,u,L)$ be a Lie progression of dimension $d$ in $C$-upper-triangular form. Suppose that $P$ is not $m$-proper. Then there exists a finite normal subgroup $H\lhd\langle P\rangle$ satisfying $H\subset P^{O_{C,d,m}(1)}$, and an $m$-proper Lie coset progression $P'=P_\Lie(\langle P\rangle,H,G',\g',\Gamma',\pi',e',u',L')$ of dimension $d'<d$ in $O_d(1)$-upper-triangular form such that $(u';L')$ is in $O_d(1)$-upper-triangular form, such that $G'$ is a quotient of $G$, and such that
\begin{equation}\label{eq:bilu.main}
HP\subset P'\subset HP^{O_{C,d,m}(1)}.
\end{equation}
\end{prop}

Proposition \ref{prop:bilu.main} essentially follows from a nilpotent version of the abelian argument that we described in the introduction. Key to this argument is the following inductive step, in which we use the geometry-of-numbers arguments developed in \cref{sec:geom-of-nos} to keep control over a Lie progression after projecting it to one of lower dimension. Here, and throughout this section, given an element $z$ of a Lie algebra $\g$ we write $z_\R$ for the one-dimensional subspace of $\g$ spanned by $z$.

\begin{prop}\label{prop:bilu}
Let $e_1,\ldots,e_d$ be a basis for a Lie algebra $\g$, and write $\Lambda$ for the lattice generated by $e_1,\ldots,e_d$. Suppose that $L_1,\ldots,L_d$ are integer lengths such that $(e;L)$ is in $C$-upper-triangular form. Let $z\ne0$ be a central element of $B_\Z(e;mL)$, and let $\varphi:\g\to\g/z_\R$ be the projection homomorphism. Then there exists a generating set $e'_1,\ldots,e'_{d-1}$ for $\varphi(\Lambda)$ and lengths $L'_1,\ldots,L'_{d-1}$ such that $B_\Z(e';L')$ is in $1$-upper-triangular form, and such that
\[
\varphi(B_\Z(e;L))\subset B_\Z(e';L')\subset\varphi(B_\Z(e;O_{C,d,m}(L))).
\]
\end{prop}

\begin{proof}
We may assume that $z$ is unimodular with respect to $\Lambda$, and hence complete $z$ to a basis of $\Lambda$. We define $\|\,\cdot\,\|$ to be the Euclidean norm with respect to which this basis is orthonormal.

Write $B=\varphi(B_\R(e;L))$, and note that on multiplying the lengths $L$ by constants depending only on $C,d$ we may assume that $[B,B]\subset B$ and that $B$ is strictly thick with respect to $\varphi(\Lambda)$. We also have $[\Lambda,\Lambda]\subset\Lambda$, and so Proposition \ref{prop:nilp.box} implies that there exists a basis $e'_1,\ldots,e'_{d-1}$ for $\varphi(\Lambda)$ and lengths $L'_1,\ldots,L'_{d-1}$ such that $B_\Z(e';L')$ is in $1$-upper-triangular form, and such that
\begin{equation}\label{eq:bilu.1}
B\subset B_\R(e';L')\subset O_{C,d}(1)B.
\end{equation}
We claim in addition that
\begin{equation}\label{eq:proj.size}
|B_\Z(e;L)|\ll\max\left\{1,\|z\|_{B_\R(e;L)}^{-1}\right\}\cdot|\varphi(B_\Z(e;L))|.
\end{equation}
Indeed, if $u,v\in  B_\Z(e;L)$ satisfy $\varphi(u)=\varphi(v)$ then we have $u-v\in z_\R\cap B_\Z(e;2L)$. However, the fact that $z$ is unimodular implies that
\[
|z_\R\cap B_\Z(e;2L)|\ll\max\left\{1,\|z\|_{B_\R(e;L)}^{-1}\right\},
\]
and so each element of $\varphi(B_\Z(e;L))$ has at most that number of preimages in $B_\Z(e;L)$, and \eqref{eq:proj.size} is proved.

We then have
\begin{align*}
|B_\Z(e';L')| &\ll_d\vol(B_\R(e';L')) \\
        &\ll_{C,d}\vol(B)      &\text{by \eqref{eq:bilu.1}} \\
        &\ll_d\|z\|_{B_\R(e;L)}\vol(B_\R(e;L))      &\text{by Lemma \ref{lem:rom.5}} \\
        &\le\|z\|_{B_\R(e;L)}\cdot|B_\Z(e;L)| \\
        &\ll\max\{1,\|z\|_{B_\R(e;L)}\}\cdot|\varphi(B_\Z(e;L))|   &\text{by \eqref{eq:proj.size}} \\
        &\le m|\varphi(B_\Z(e;L))|,
\end{align*}
and so the desired result follows from Proposition \ref{prop:rom.8} and the first inclusion of \eqref{eq:bilu.1}.
\end{proof}

The next lemma shows that given a progression that is not proper, the lack of properness is witnessed at some \emph{central} element, which will ultimately allow us to apply \cref{prop:bilu}.

\begin{lemma}\label{lem:collision}
Let $e_1,\ldots,e_d$ be a basis of the Lie algebra $\g$ of a connected, simply connected nilpotent Lie group $G$, and let $L_1,\ldots,L_d$ be positive integers such that $(e;L)$ is in $C$-upper-triangular form. Suppose that $\Gamma=\exp\langle e_1,\ldots,e_d\rangle$ is a subgroup of $G$ and that $H\lhd\Gamma$ is a normal subgroup of $\Gamma$. Suppose that there exist $\ell_1,\ldots,\ell_d,\ell_1',\ldots,\ell_d'$ with $|\ell_i|,|\ell_i'|\le mL_i$ for each $i$ and $\ell_i\ne\ell_i'$ for at least one $i$ such that
\[
x=\exp(\ell_1e_1+\ldots+\ell_de_d)\exp(-\ell'_1e_1-\ldots-\ell'_de_d)\in H.
\]
Then there exists some non-zero $z\in B_\Z(e;O_{C,d,m}(L))\cap\z(\g)$ such that $\exp z\in H$.
\end{lemma}
\begin{proof}
Note that $x\ne1$ since $\ell_1e_1+\ldots+\ell_de_d\ne\ell'_1e_1+\ldots+\ell'_de_d$. It follows from Lemmas \ref{lem:upper-tri.dilate} and \ref{lem:QL.group} that
\[
x\in\exp B_\Z(e;O_{C,d,m}(L))\cap H.
\]
If $x$ is not central then there is some $i$ such that $[x,\exp e_i]\ne1$. The normality of $H$ and further applications of Lemmas \ref{lem:upper-tri.dilate} and \ref{lem:QL.group} imply that $[x,\exp e_i]\in\exp B_\Z(e;O_{C,d,m}(L))\cap H$. We may therefore, following Tao \cite[\S4]{tao.growth}, replace $x$ by $[x,\exp e_i]$ and repeat until we have some non-trivial central element $x'\in\exp B_\Z(e;O_{C,d,m}(L))\cap H$ (noting that this will require at most $d$ repetitions). The lemma is then satisfied by taking $z=\log x'$.
\end{proof}

By repeatedly applying \cref{prop:bilu,lem:collision} we arrive at the following result.

\begin{prop}\label{prop:biluGroup}
Let $N$ be a group, let $G$ be a connected, simply connected nilpotent Lie group, let $e_1,\ldots,e_d$ be a basis for the corresponding Lie algebra $\g$, and suppose that $B_\Z(e;L)$ is in $C$-upper-triangular form. Write $\Lambda=\langle e_1,\ldots,e_d\rangle$, and suppose that $\Gamma=\exp\Lambda$ is a group. Let $m,m'>0$, and suppose further that $\pi:\Gamma\to N$ is a surjective homomorphism with respect to which $B_\Z(e;L)$ is not $m$-proper.

Then there exist $d'<d$; positive integers $L'_1,\ldots,L'_{d'}$; a connected, simply connected nilpotent Lie group $G'$ that is a quotient of $G$ whose Lie algebra $\g'$ has a basis $e'_1,\ldots,e'_{d'}$ such that $(e';L')$ is in $1$-upper-triangular form and such that, writing $\Lambda'=\langle e'_1,\ldots,e'_{d'}\rangle$, the image $\Gamma'=\exp\Lambda'$ is a group; a finite normal subgroup $H\lhd N$ satisfying
\[
H\subset\pi(\exp B_\Z(e;O_{C,d,m,m'}(L)));
\]
and a homomorphism $\pi':\Gamma'\to N/H$ with respect to which $B_\Z(e';L')$ is $m'$-proper such that
\[
H\pi(\exp B_\Z(e;rL))\subset\pi'(\exp B_\Z(e';rL'))\subset H\pi(\exp B_\Z(e;O_{C,d,m,m'}(rL)))
\]
for every $r\in\N$.
\end{prop}

\begin{proof}First note that we may assume $m=m'$. Indeed, if $m>m'$ then proving that $B_\Z(e';L')$ is $m$-proper with respect to $\pi'$ certainly implies that it is $m'$-proper with respect to $\pi'$, whilst if $m<m'$ and $B$ is not $m$-proper with respect to $\pi$ then neither is it $m'$-proper with respect to $\pi$.

We will give recursive definitions of various sequences of Lie algebras, Lie groups, groups, homomorphisms and elements. To start the process, we denote $G_0=G$, $\g_0=\g$, $\Lambda_0=\Lambda$, $\Gamma_0=\Gamma$ and $e^{(0)}_i=e_i$ for $i=1,\ldots,d$, and set $H_0=\{1\}$. For the recursive step, suppose we have already defined a normal subgroup $H_j\lhd N$, a connected, simply connected nilpotent Lie group $G_j$ that is a homomorphic image of $G$ with Lie algebra $\g_j$ with basis $e^{(j)}_1,\ldots,e^{(j)}_{d-j}$, and lengths $L^{(j)}_1,\ldots,L^{(j)}_{d-j}$ such that $(e^{(j)};L^{(j)})$ is in $1$-upper-triangular form (or $C$-upper-triangular form for $j=0$) and such that, writing $\Lambda_j=\langle e^{(j)}_1,\ldots,e^{(j)}_{d-j}\rangle$, the set $\Gamma_j=\exp\Lambda_j$ forms a group. Suppose moreover that we have defined a homomorphism $\pi_j:\Gamma_j\to N/H_j$.

We stop this recursive process if $B_\Z(e^{(j)};L^{(j)})$ is $m$-proper with respect to $\pi_j$. If $B_\Z(e^{(j)};L^{(j)})$ is not $m$-proper with respect to $\pi_j$ then Lemma \ref{lem:collision} implies that there exists some non-zero
\[
z\in B_\Z(e^{(j)};O_{C,d,m}(L^{(j)}))\cap\z(\g_j)\cap\log\ker\pi_j.
\]
In that case, define $H_{j+1}$ to be the pullback to $N$ of $\pi_j(\Gamma_j\cap\exp z_\R)$, noting that $H_{j+1}$ is normal (since $z$ is central) and that
\begin{equation}\label{eq:H.controlled}
H_{j+1}/H_j\subset\pi_j(\exp B_\Z(e^{(j)};O_{C,d,m}(L^{(j)}))).
\end{equation}
Define $G_{j+1}=G_j/(\exp z_\R)$, noting that this quotient is connected and simply connected and a homomorphic image of $G$, write $\Phi_j:G_j\to G_{j+1}$ for the projection homomorphism, and write $\Gamma_{j+1}=\Phi_j(\Gamma_j)$. Similarly, define $\g_{j+1}=\g_j/z_\R$, write $\varphi_j:\g_j\to\g_{j+1}$, and write $\Lambda_{j+1}=\varphi_j(\Lambda_j)$. Define $\pi_{j+1}$ so that the diagram
\[
\begin{CD}
 \Gamma_j                      @>\pi_j>>           N/H_j\\
@V\Phi_{j+1}|_{\Gamma_j}VV               @VVV\\
\Gamma_{j+1}     @>\pi_{j+1}>>    N/H_{j+1}
\end{CD}
\]
commutes.

Proposition \ref{prop:bilu} implies that there exists a basis $e^{(j+1)}_1,\ldots,e^{(j+1)}_{d-j-1}$ for $\g_{j+1}$ and lengths $L^{(j+1)}_i$ such that $(e^{(j+1)};L^{(j+1)})$ is in $1$-upper-triangular form, and such that
\begin{align*}
\varphi_{j+1}(B_\Z(e^{(j)};L^{(j)}))&\subset B_\Z(e^{(j+1)};L^{(j+1)})\\
    &\subset\varphi_{j+1}(B_\Z(e^{(j)};O_{C,d,m}(L^{(j)}))).
\end{align*}
Note that this is where the geometry of numbers enters the argument.

Since the dimension of $\g_j$ drops at each stage, this process necessarily terminates for some $j\le d$, at which point $B_\Z(e^{(j)};L^{(j)})$ is $m$-proper with respect to $\pi_j$ by definition. We define $\varphi:\g_0\to\g_j$ and $\Phi:G_0\to G_j$ via $\varphi=\varphi_j\circ\cdots\circ\varphi_1$ and $\Phi=\Phi_j\circ\cdots\circ\Phi_1$, so that the diagram 
\begin{equation}\label{eq:diagram}
\begin{CD}
\Lambda_0        @>\exp>>        \Gamma_0       @>\pi>>      N\\
@V\varphi|_{\Lambda_0}VV     @V\Phi|_{\Gamma_0}VV       @VVV\\
\Lambda_j         @>\exp>>        \Gamma_j        @>\pi_j>>   N/H_j
\end{CD}
\end{equation}
commutes. Abbreviating $H=H_j$, $\pi'=\pi_j$, $e_i'=e^{(j)}_i$ and $L_i'=L^{(j)}_i$, and setting $d'=d-j$, we therefore have
\begin{align*}
\varphi(B_\Z(e;L^{(0)}))&\subset B_\Z(e';L') \label{eq:dim.red}\\
&\subset\varphi(B_\Z(e;O_{C,d,m}(L^{(0)}))),
\end{align*}
and so commutativity of the diagram \eqref{eq:diagram} implies that
\[
H\pi(\exp B_\Z(e;rL))\subset\pi'(\exp B_\Z(e';rL'))\subset H\pi(\exp B_\Z(e;O_{C,d,m}(rL)))
\]
for every $r\in\N$, as required.
\end{proof}

\begin{proof}[Proof of Proposition \ref{prop:bilu.main}]
For each $d\in\N$ and $C>0$, let $k_{C,d}\in\N$ be large enough to replace the constants implicit in the containments \eqref{eq:dilates.concl.1} and \eqref{eq:dilates.concl.2}, chosen such that $k_{C,d}$ is increasing in $d$. Defining $p_{C,d}$ as in Proposition \ref{prop:dilates}, it follows immediately from that proposition that $B_\Z(e;L)$ is not $p_{C,d}(m)$-proper with respect to $\pi$, and so Lemma \ref{lem:upper-tri.dilate} and Proposition \ref{prop:biluGroup} imply that there exist $d'<d$; positive integers $L'_1,\ldots,L'_{d'}$; a nilpotent Lie group $G'$ that is a homomorphic image of $G$ whose Lie algebra $\g'$ has a basis $e'_1,\ldots,e'_{d'}$ such that $(e';L')$ is in $1$-upper-triangular form and such that, writing $\Lambda'=\langle e'_1,\ldots,e'_{d'}\rangle$, the image $\Gamma'=\exp\Lambda'$ is a group; a finite normal subgroup $H\lhd\langle P\rangle$ satisfying
\begin{equation}\label{eq:bilu.H}
H\subset\pi(\exp B_\Z(e;O_{C,d,m}(L)));
\end{equation}
and a homomorphism $\pi':\Gamma'\to N/H$ with respect to which $\exp B_\Z(e';L')$ is $p_{1,d}(k_{1,d}m)$-proper and such that
\begin{equation}\label{eq:bilu.group}
H\pi(\exp B_\Z(e;k_{1,d}rL))\subset\pi'(\exp B_\Z(e';rL'))\subset H\pi(\exp B_\Z(e;O_{C,d,m}(rL)))
\end{equation}
for every $r\in\N$. Note then that
\begin{align*}
H\pi(P_\ord(u;L))&\subset H\pi(\exp B_\Z(e;k_{C,d}L)) & \text{(by Proposition \ref{prop:dilates})} \\
          &\subset\pi'(\exp B_\Z(e';L')) & \text{(by \eqref{eq:bilu.group})} \\
          &\subset\pi'(P_\ord(u';k_{1,d}L')) & \text{(by Proposition \ref{prop:dilates})} \\
          &\subset\pi'(\exp B_\Z(e';O_d(L'))) & \text{(by Lemma \ref{lem:upper-tri.dilate} and Proposition \ref{prop:dilates})} \\
          &\subset H\pi(\exp B_\Z(e;O_{C,d,m}(L))) & \text{(by \eqref{eq:bilu.group})} \\
          &\subset H\pi(P_\ord(u;O_{C,d,m}(L))) & \text{(by Lemma \ref{lem:upper-tri.dilate} and Proposition \ref{prop:dilates})} \\
          &\subset H\pi(P_\ord(u;L))^{O_{C,d,m}(L)} & \text{(by \eqref{eq:lengths.powers})},
\end{align*}
which proves \eqref{eq:bilu.main} with $P'=\pi'(P(u';k_{1,d}L'))$. It follows from Lemma \ref{lem:upper-tri.dilate} that $(e',k_{1,d}L')$ is in $O_d(1)$-upper-triangular form. It follows from Proposition \ref{prop:dilates} that $P_\ord(u',L')$ is $k_{1,d}m$-proper with respect to $\pi'$ and in $O_d(1)$-upper-triangular form, and hence that $P(u';k_{1,d}L')$ is $m$-proper with respect to $\pi'$ and, by Lemma \ref{lem:upper-tri.dilate}, in $O_d(1)$-upper-triangular form. Finally, it follows from Proposition \ref{prop:dilates}, Lemma \ref{lem:upper-tri.dilate} and \eqref{eq:lengths.powers} and \eqref{eq:bilu.H} that $H\subset P^{O_{C,d,m}(1)}$.
\end{proof}

\begin{proof}[Proof of Theorem \ref{thm:nilp.bilu.detailed}]
First, note that it is sufficient to prove the theorem with the conclusion that both tuples $(e;L)$ and $(u;L)$ are in $O_{r,s}(1)$-upper-triangular form. Indeed, if we obtain such a conclusion and $P_\ord(u;L)$ is $m$-proper then upon multplying the lengths $L_i$ by constants $k_i$ depending only on $C,r,s$ we can put $(u;L)$ in $C$-upper-triangular form whilst ensuring that $P_\ord(u;L)$ is $(m/\max_ik_i)$-proper and leaving $(e;L)$ in $O_{C,r,s}(1)$-upper-triangular form.

Write $N=\langle P\rangle$. By definition (see also Remark \ref{rem:free.image}) there exists a free nilpotent progression $P_\ord(u;L)$ of total rank $d\ll_{r,s}1$ and a homomorphism $\pi:\langle P_\ord(u;L)\rangle\to N$ such that $P=\pi(P_\ord(u;L))$. Proposition \ref{prop:nilp.prog.upper.tri} implies that $P_\ord(u;L)$ is in $O_{r,s}(1)$-upper-triangular form. Write $\Gamma=\langle P_\ord(u;L)\rangle$, and recall from Section \ref{sec:coords} that we may assume that $\Gamma$ is a subgroup of the free nilpotent Lie group $G$ of rank $r$ and step $s$. Denote the Lie algebra of $G$ by $\g$, and write $e_i=\log u_i$ for each $i$.

Proposition \ref{prop:nilbox.upper-tri} followed by Lemma \ref{lem:dilates.subgroup} implies that there exist integers $Q_1,\ldots,Q_d\ll_{r,s}1$ such that $(Q_1e_1,\ldots,Q_de_d;L)$ is in $O_{r,s}(1)$-upper-triangular form and such that $\exp\langle Q_1e_1,\ldots,Q_de_d\rangle$ is a group. Theorem \ref{thm:collected.unique} and Proposition \ref{prop:nilp.prog.upper.tri} imply that we may apply Lemma \ref{lem:coset.reps} to $P_\ord(u;L)$, and this shows that there exists $X\subset P_\ord(u;L)^{O_{r,s}(1)}$ with $|X|\ll_{r,s}1$ such that
\begin{equation}\label{eq:nilp.bilu.cosets}
P_\ord(u;L)\subset XP_\ord(u_1^{Q_1},\ldots,u_d^{Q_d};O_{r,s}(L)).
\end{equation}
Proposition \ref{prop:dilates} implies that $P_\ord(u_1^{Q_1},\ldots,u_d^{Q_d};L)$ is in $O_{r,s}(1)$-upper triangular form, and it is trivially the case that $P_\ord(u_1^{Q_1},\ldots,u_d^{Q_d};L)\subset P_\ord(u;L)^{O_{r,s}(1)}$.

The theorem therefore holds if $P_\ord(u_1^{Q_1},\ldots,u_d^{Q_d};L)$ is $m$-proper with respect to $\pi$; if it is not then the theorem follows from Proposition \ref{prop:bilu.main}.
\end{proof}

\begin{proof}[Proof of Theorem \ref{thm:ab.Bilu}]
This is immediate from the abelian cases of Theorem \ref{thm:nilp.Frei} and Proposition \ref{prop:bilu.main}. We emphasise that the abelian case of Theorem \ref{thm:nilp.Frei} is originally due to Green and Ruzsa \cite[Theorem 1.1]{green-ruzsa}, with a slight modification by Breuillard and Green (see \cite[Theorem 1.3$'$]{bg}).
\end{proof}

\begin{remark*}Theorem \ref{thm:ab.Bilu} does not need the full strength of \cref{prop:bilu.main}, the proof of which is much simpler in the abelian case. We leave it to the interested reader to work out the details of this simpler proof, and thus to shorten the proof of Theorem \ref{thm:ab.Bilu}.
\end{remark*}

\section{Sets of polynomial growth in terms of progressions}\label{sec:poly.growth.progs}
The basic idea behind our proof of Theorem \ref{thm:persitstence.poly.growth} is to control the growth of $S^m$ in terms of the growth of a certain nilprogression of bounded rank and step. In its simplest form, the tool that allows us to do this is the following result, which essentially appeared in \cite{tao.growth} and was implicit in \cite{bt}.
\begin{prop}\label{prop:inverse}
Let $M,D>0$, and let $S$ be a finite symmetric generating set for a group $G$ such that $1\in S$. Then there exists $N=N_{M,D}$ such that if $|S^n|\le Mn^D|S|$ for some $n\ge N$ then there exist $X\subset S^{O_D(1)}$ with $|X|\ll_D1$ and a $1$-proper ordered coset progression $HP$ of rank at most $O_D(1)$ in $O_D(1)$-upper-triangular form such that $XHP^r\subset S^{rn}\subset XHP^{O_D(r)}$ for every $r\in\N$.
\end{prop}
\begin{remark*}
In the converse direction, it follows from Lemma \ref{lem:lengths.powers} that if there exist a finite set $X$ and an $m$-proper ordered coset progression $HP$ of rank $d$ in $C$-upper-triangular form such that $S^n\subset XHP^k$ then $|S^n|\ll_{C,d,m,k,|X|}|HP|$. Proposition \ref{prop:inverse} is therefore already an example of an inverse theorem for sets of polynomial growth with matching direct theorem of the type described in \cite[\S1]{tao.growth}.
\end{remark*}
\begin{remark}
At the expense of making the other bounds dependent on $m$ and $C$ one could, for any $m$ and $C$, insist that the ordered coset progression in Proposition \ref{prop:inverse} be $m$-proper and in $C$-upper-triangular form.
\end{remark}

We actually need the following more-detailed version of \cref{prop:inverse}. Given an $m$-proper ordered coset progression $HP$ with $P=P_\ord(u_1,\ldots,u_d;L)$, every element $x\in HP_\ord(u;mL)$ has, by definition, a unique representation
\[
x=hu_1^{x_1}\cdots u_d^{x_d}
\]
with $h\in H$ and $|x_i|\le mL$; in this case we call $x_i$ the \emph{$u_i$-coordinate} of $x$ with respect to $HP$.

\begin{prop}\label{prop:inverse+}
Let $M,D>0$, let $m\ge1$, let $k\in\N$, and let $S$ be a finite symmetric generating set for a group $G$ such that $1\in S$. Then there exists $N=N_{\cD,m,k}$ such that if $n\ge\max\{N,NM\}$ and
\begin{equation}\label{eq:poly.growth}
|S^n|\le Mn^D|S|
\end{equation}
then there exist $t=t_\cD\in\N$ and
\begin{equation}\label{eq:X.in.S^C}
X\subset S^t
\end{equation}
containing $1$ with $|X|\ll_D1$, and there exist $C=C_\cD>0$ and an $m$-proper ordered coset progression $HP(u;L)$ of rank $d\ll_\cD1$ in $C$-upper-triangular form, such that
\begin{equation}\label{eq:XHP.proper}
xHP^k\cap yHP^k=\varnothing
\end{equation}
for every $x,y\in X$ with $x\ne y$, and such that
\begin{equation}\label{eq:inverse}
XHP^r\subset S^{rn}\subset XHP^{O_{\cD,m,k}(r)}
\end{equation}
for every $r\in\N$. Moreover, there exists $c\gg_{\cD,m,k}1$ such that $cn\in2\Z$ and
\begin{equation}\label{eq:point.where.we.assume.minimality}
S^{cn}\subset XHP\subset S^n,
\end{equation}
and such that if we define $\zeta(i)$ as before Lemma \ref{lem:lengths.powers} then for every generator $u_i$ of $P$ with $\zeta(i)=1$ there exists $s\in S^{cn}$ and $x\in X$, and $p\in HP$ with non-zero $u_i$-coordinate, such that $s=xp$.
\end{prop}

We start with a lemma showing that polynomial growth of $S^n$ implies small doubling of some $S^k$ with $k\le n$. This technique is completely standard, having been used in Gromov's proof of his polynomial growth theorem \cite{gromov}, for example, but we nonetheless include the following lemma in order to have a precise record of the dependence of the constants on one another.

\begin{lemma}\label{lem:poly.pigeon}
Let $M,D>0$, let $\alpha,\beta\in(0,1)$, and let $q\in\N$. Then there exists $N=N_\alpha$ such that if $S$ is a finite subset of a group with
\begin{equation}\label{eq:poly.pigeon}
|S^n|\le Mn^D|S|
\end{equation}
for some $n\ge\max\{N,M(q\beta^{-1})^\frac{D+1}{1-\alpha}\}$ then there exists $k\in\N$ satisfying $n^\alpha<k<\beta n$ such that $|S^{qk}|\le q^{\frac{D+1}{1-\alpha}}|S^k|$.
\end{lemma}
\begin{proof}
Suppose \eqref{eq:poly.pigeon} holds for a given $n$. Provided $n$ is large enough in terms of $\alpha$, on increasing $\alpha$ slightly if necessary we may assume that $n^\alpha\in\N$. Fix $z\in\Z$ such that $q^z\le \beta n^{1-\alpha}<q^{z+1}$. Condition \eqref{eq:poly.pigeon} implies in particular that $|S^n|\le Mn^D|S^{n^\alpha}|$, and so if $|S^{q^{\ell+1}n^\alpha}|>q^\frac{D+1}{1-\alpha}|S^{q^{\ell}n^\alpha}|$ for every $\ell\le z$ then we have
\begin{align*}
|S^n|&\ge|S^{q^zn^\alpha}|\\
    &>q^{\frac{z(D+1)}{1-\alpha}}|S^{n^\alpha}|\\
    &=\left(\frac{\beta}{q}\right)^\frac{D+1}{1-\alpha}\left(\beta^{-1}q^{z+1}\right)^\frac{D+1}{1-\alpha}|S^{n^\alpha}|\\
    &>\left(\frac{\beta}{q}\right)^\frac{D+1}{1-\alpha}n^{D+1}|S^{n^\alpha}|,
\end{align*}
and hence $n<M(q\beta^{-1})^\frac{D+1}{1-\alpha}$.
\end{proof}

It is then relatively straightforward to express the growth of this set $S^k$ in terms of a nilpotent progression using an argument of Breuillard and the second author, as follows.

\begin{prop}\label{prop:inverse''}
Let $M,D>0$. Then there exists $N=N_\cD$ such that if $S$ is a finite symmetric generating set for a group $G$ such that $1\in S$ and $|S^n|\le Mn^D|S|$ for some $n\ge\max\{N,5^{2D+2}M\}$ then there exists $k\in\Z$ with $n^{1/2}\le k\le n$, a set $X\subset G$ with $|X|\ll_\cD1$, and a nilpotent coset progression $HP$ of rank and step at most $O_\cD(1)$ such that $S^{rk}\subset XHP^r\subset S^{O_\cD(rk)}$ for every $r\in\N$.
\end{prop}
\begin{proof}
It follows from Lemma \ref{lem:poly.pigeon} that, provided $n$ is larger than some absolute constant, if $n\ge5^{2D+2}M$ then there exists $k$ with $n^{1/2}\le k\le n$ such that $|S^{5k}|\le5^{2D+2}|S^k|\le5^{2\cD+2}|S^k|$. Provided that $n$ is large enough in terms of $\cD$, it then follows from \cite[Proposition 2.9]{bt} that there exist subgroups $H_0\lhd\Gamma <G$ such that $\Gamma/H_0$ is nilpotent of step at most $O_\cD(1)$, an $O_\cD(1)$-approximate group $A\subset S^{8k}\cap\Gamma$ containing $H_0$, and a set $X\subset S^k$ with $|X|\ll_\cD1$ such that for every $r\in\N$ we have $S^{rk}\subset XA^r$. Applying Theorem \ref{thm:nilp.Frei} to the image of $A$ in $\Gamma/H_0$, there therefore exists a nilpotent coset progression $HP\subset A^{O_\cD(1)}$ of rank and step at most $O_\cD(1)$ such that $A\subset HP$, and hence $S^{rk}\subset XHP^r\subset S^{O_\cD(rk)}$, as required.
\end{proof}

\begin{remark}
The proof of \cite[Proposition 2.9]{bt} uses Theorem \ref{thm:bgt.noprog}. This is the only place in our proofs of Proposition \ref{prop:inverse} and Theorem \ref{thm:persitstence.poly.growth} that we use Theorem \ref{thm:bgt.noprog}, and hence the only source of ineffectiveness in these results.
\end{remark}

Of course, it is really the set $S^n$ itself whose growth we wish to express in terms of that of a progression. We can almost do this using Proposition \ref{prop:inverse''}, in that choosing $r_0$ so that $(r_0-1)k<n\le r_0k$ we can control powers of $S^n$ in terms of powers of $HP^{r_0}$. The following proposition allows us to replace $HP^{r_0}$ with another nilpotent coset progression.

\begin{prop}\label{prop:powers.are.nilprogs}
If $P=P_\ord(u_1,\ldots,u_d;L)$ is a nilpotent progression and $r\in\N$ then there is a nilpotent progression $P_r$ on the same generators as $P$ such that $P^r\subset P_r\subset P^{O_d(r)}$. Moreover, if $P$ is $m$-proper then we may take $P_r$ to be $\Omega_{d,r}(m)$-proper.
\end{prop}

\begin{proof}
Write $x_1,\ldots,x_{d'}$ for the ordered set of letters in which $u_1,\ldots,u_d$ are basic commutators. Proposition \ref{prop:nilp.prog.upper.tri} implies that $P$ is in $O_d(1)$-upper-triangular form, so we may apply Lemma \ref{lem:lengths.powers}. By Lemma \ref{lem:weights.agree}, the conclusion of Lemma \ref{lem:lengths.powers} implies that there exists $\gamma\in\N$ with $\gamma\ll_d1$ such that $P^r\subset P_\ord(u;(\gamma r)^{|\chi|}L)$. This is a nilpotent progression on the $x_i$ by definition, and is certainly $\Omega_{d,r}(m)$-proper if $P$ is $m$-proper.

It remains to show that $P_\ord(u;(\gamma r)^{|\chi|}L)\subset P^{O_d(r)}$, which we do following the proof of \cite[Proposition 3.10 (1)]{bt}. Write $B(x;L)=\cup_{i=1}^{d'}\{x_i^{\ell_1}:|\ell_1|\le L_i\}$. Given a basic commutator $u_i$ and $\ell\le(\gamma r)^{|\chi(u_i)|}L^{\chi(u_i)}$, \cite[Lemma C.2]{nilp.frei} says that $u_i^\ell\in B(x;\gamma rL)^{O_d(1)}$. However, by \eqref{eq:n.prog.lengths} we have $(\gamma r)^{|\chi(u_i)|}L^{\chi(u_i)}=(\gamma r)^{|\chi(u_i)|}L_i$, so in fact this gives us
\begin{align*}
P_\ord(u;(\gamma r)^{|\chi|}L)&\subset B(x;\gamma rL)^{O_d(1)}\\
    &\subset B(x;L)^{O_d(\gamma r)}\\
    &\subset P^{O_d(r)},
\end{align*}
as required.
\end{proof}

Combining Propositions \ref{prop:inverse''} and \ref{prop:powers.are.nilprogs}, we are now able to control powers of $S^n$ in terms of powers of a nilpotent progression, as follows.

\begin{prop}\label{prop:inverse'}
Let $M,D>0$. Then there exists $N=N_\cD$ such that if $S$ is a finite symmetric generating set for a group $G$ such that $1\in S$ and $|S^n|\le Mn^D|S|$ for some $n\ge\max\{N,5^{2D+2}M\}$ then there exists a set $X\subset G$ with $|X|\ll_\cD1$ and a nilpotent coset progression $HP$ of rank and step at most $O_\cD(1)$ such that $S^{rn}\subset XHP^r\subset S^{O_\cD(rn)}$ for every $r\in\N$.
\end{prop}
\begin{proof}
Apply Proposition \ref{prop:inverse''} and choose $r_0$ such that $(r_0-1)k<n\le r_0k$, so that
\[
S^{r'n}\subset XHP^{r'r_0}\subset S^{O_\cD(r'n)}
\]
for every $r'\in\N$. Proposition \ref{prop:powers.are.nilprogs} then implies that we may replace $HP^{r_0}$ with a nilpotent coset progression of the same rank and step as $HP$, and so the proposition is proved.
\end{proof}

The argument of \cite[Proposition 2.9]{bt} underpinning Proposition \ref{prop:inverse''}, and hence ultimately Proposition \ref{prop:inverse'}, exploits the fact that the elements of the set $X$ belong to distinct left-cosets of the group $\Gamma=\langle HP\rangle$. However, it turns out one can run similar arguments under a weaker `local' version of this hypothesis, in which the elements of $X$ merely belong to distinct left-translates of $HP^{O(1)}$. The following lemma shows that this local version of the hypothesis is in fact very general.

\begin{lemma}\label{lem:cosets.far.apart}
Let $k\in\N$. Let $HP_0=HP(u;L)$ be an ordered coset progression of rank $d$ in $C$-upper-triangular form in a group $G$, and let $X_0$ be a finite subset of $G$. Then there exists $X\subset X_0$, and an ordered progression $P\subset G$ on the same generators as $P_0$ such that $HP$ is in $O_{C,d,k,|X_0|}(1)$-upper-triangular form, such that
\[
X_0HP_0\subset XHP\subset X_0HP_0^{O_{C,d,k,|X_0|}(1)},
\]
and such that
\[
xHP^k\cap yHP^k=\varnothing
\]
for every pair $x,y\in X$ with $x\ne y$. Moreover, if $HP_0$ is a nilpotent coset progression then we may also take $HP$ to be a nilpotent coset progression, and if $HP_0$ is $m$-proper then we may take $HP$ to be $\Omega_{C,d,k,|X|}(m)$-proper.
\end{lemma}
\begin{proof}
We proceed by induction on $|X_0|$, noting that the result is trivial if $X_0$ is a singleton. If $xHP_0^k\cap yHP_0^k=\varnothing$ for every pair $x,y\in X_0$ with $x\ne y$ then we may take $X=X_0$ and $P=P_0$. If not then there exist distinct elements $x,y\in X_0$ such that $xHP_0^k$ and $yHP_0^k$ have non-trivial intersection, which implies that $x\in yHP_0^kP_0^{-k}$, and hence by \eqref{eq:prog.inverse} that $x\in yHP_0^{dk+k}$. In particular, setting $X_1=X_0\backslash\{x\}$ we have $X_0HP_0\subset X_1HP_0^{dk+k+1}$. 

If $HP_0$ is a nilpotent coset progression then Proposition \ref{prop:powers.are.nilprogs} implies that there exists a nilpotent coset progression $HP_1$ on the same generators such that $HP_0^{dk+k+1}\subset HP_1\subset HP_0^{O_d(k)}$ and such that $HP_1$ is $\Omega_{d,k}(m)$-proper if $HP_0$ is $m$-proper. Moreover, $HP_1$ is in $O_d(1)$-upper-triangular form by Proposition \ref{prop:nilp.prog.upper.tri}, and so the lemma follows by induction.

If $HP_0$ is not a nilpotent coset progression then it nonetheless follows from Lemma \ref{lem:lengths.powers} and \eqref{eq:lengths.powers} that there exists $r\ll_{C,d,k}1$ such that if we set $P_1=P_\ord(u;rL)$ then $X_0HP_0\subset X_1HP_1\subset X_0HP_0^{O_{C,d,k}(1)}$. Note that $HP_1$ is in $rC$-upper-triangular form, and is $m/r$-proper if $HP_0$ is $m$-proper. Again, the lemma therefore follows by induction.
\end{proof}

The following lemma can be thought of as a local version of \cite[Lemma 2.7]{bt}, with the assumption \eqref{eq:local.coset} replacing the stronger assumption that $x\notin y\langle A\rangle$ for every pair $x,y\in X$ with $x\ne y$.

\begin{lemma}\label{lem:coset.reps.local}
Let $q\in\N$. Let $S$ be a finite generating set for a group $G$ with $1\in S$, and let $X,A$ be subsets of $G$ containing $1$ such that $|X|<q$ and
\begin{equation}\label{eq:Sn.in.XA}
S^q\subset XA.
\end{equation}
Suppose that
\begin{equation}\label{eq:local.coset}
x\notin y(AA^{-1})^3
\end{equation}
for every pair $x,y\in X$ with $x\ne y$. Then there exists $X'\subset S^{|X|}\cap XA$ containing the identity such that $|X'|\le |X|$ and $S^q\subset X'A^{-1}A$.
\end{lemma}
\begin{proof}
Set $X_j=\{x\in X:xA\cap S^j\ne\varnothing\}$ for $j=1,\ldots,q$, noting that
\begin{equation}\label{eq:SrXr}
S^j\subset X_jA
\end{equation}
for every such $j$. Since $1\in S$, the set $S^j$ is non-decreasing in $j$, and hence so is the set $X_j$. Since $|X|<q$ there is therefore some $r\le|X|$ satisfying $X_r=X_{r+1}$. For each $x\in X_r$, we may by definition pick $x'\in xA\cap S^r$, taking in particular $1'=1$. Write $X'$ for the set of $x'$ we have chosen, noting that $X'\subset S^{|X|}\cap X_rA$. Note also that $X_r\subset X'A^{-1}$, which combines with \eqref{eq:SrXr} and the definition of $r$ to imply that
\begin{equation}\label{eq:Sr+1X'}
S^{r+1}\subset X'A^{-1}A.
\end{equation}
Now \eqref{eq:Sn.in.XA} implies that $S^q\cap X'(A^{-1}A)^2\subset XA\cap X'(A^{-1}A)^2$, which is in turn a subset of $XA\cap X_rA(A^{-1}A)^2$. It follows that every element $z$ of $S^q\cap X'(A^{-1}A)^2$ can be written both in the form $xa_6$ and in the form $x'a_1a_2^{-1}a_3a_4^{-1}a_5$ with $x\in X$, with $x'\in X_r$ and with $a_i\in A$. Since this implies $x=x'a_1a_2^{-1}a_3a_4^{-1}a_5a_6^{-1}$, \eqref{eq:local.coset} implies that $x=x'$, and hence $z=x'a_6\in X_rA$. Thus $S^q\cap X'(A^{-1}A)^2\subset X_rA$, and hence
\begin{equation}\label{eq:bootstrap}
S^q\cap X'(A^{-1}A)^2\subset X'A^{-1}A.
\end{equation}
We claim in addition that
\begin{equation}\label{SiX'.in.X'A2}
S^iX'\subset X'A^{-1}A
\end{equation}
for every $i\le q-r$. The case $i=1$ follows from \eqref{eq:Sr+1X'} and the fact that $X'\subset S^r$. When $1<i\le q-r$, on the other hand, we have
\begin{align*}
S^iX'&=S^{i-1}SX'\\
    &\subset S^{i-1}X'A^{-1}A &\text{(by the case $i=1$)}\\
    &\subset X'(A^{-1}A)^2 &\text{(by induction).}
\end{align*}
Since $X'\subset S^r$ and $i\le q-r$, we therefore have in particular $S^iX'\subset S^q\cap X'(A^{-1}A)^2$, and so \eqref{SiX'.in.X'A2} follows from \eqref{eq:bootstrap}, as claimed. However, we also have
\begin{align*}
S^q&\subset S^{q-r}X'A^{-1}A&\text{(by \eqref{eq:Sr+1X'})}\\
    &\subset X'(A^{-1}A)^2&\text{(by \eqref{SiX'.in.X'A2}),}
\end{align*}
and so \eqref{eq:bootstrap} gives $S^q\subset X'A^{-1}A$, as required.
\end{proof}

The following lemma can be thought of as a local version of part of \cite[Proposition 2.9]{bt}.

\begin{lemma}\label{lem:set.growth.in.terms.of.prog}
Let $q\in\N$, let $S$ be a finite generating set for a group $G$, and let $X,A$ be subsets of $G$ such that $X\subset S^q$ and $S^{2q}\subset XA$. Let $r\in\N$ with $r\ge 2$. Then $S^{rq}\subset X(A\cap S^{-q}S^{2q})^{r-1}$.
\end{lemma}
\begin{proof}
The case $r=2$ is trivial. For $r>2$ we have $S^{rq}\subset S^qX(A\cap S^{-q}S^{2q})^{r-2}$ by induction, and then since $X\subset S^q$ the desired conclusion follows from the $r=2$ case.
\end{proof}

\begin{proof}[Proof of Proposition \ref{prop:inverse+}]
Let $c\gg_{\cD,m,k}1$ and $\rho\in\N$ be constants to be chosen later such that $c<1$, such that $cn\in2\Z$ and such that $\rho\ll_\cD1$. The bound \eqref{eq:poly.growth} implies in particular that $|S^{cn}|\le Mc^{-D}(cn)^D|S|$. Provided $n$ is large enough in terms of $\cD$ and $n\ge5^{2D+2}c^{-D-1}M$, therefore, Proposition \ref{prop:inverse'} applied to $S^{cn}$ implies that there exists a set $X_0\subset G$ with $|X_0|\ll_\cD1$ and a nilpotent coset progression $H_0P_0$ of rank and step at most $O_\cD(1)$ such that $S^{cn}\subset X_0H_0P_0\subset S^{O_\cD(cn)}$. Theorem \ref{thm:nilp.bilu} then implies that there exists a set $X\subset G$ containing $1$ with $|X|\ll_\cD1$ and an $\rho m$-proper ordered coset progression $HP_1=HP_\ord(u;L)$ in $1$-upper-triangular form and of rank at most $O_\cD(1)$ such that
\begin{equation}\label{eq:before.k}
S^{cn}\subset XHP_1\subset S^{O_{\cD,m}(cn)}.
\end{equation}
By Lemma \ref{lem:cosets.far.apart}, and at the expense of weakening \eqref{eq:before.k} to $S^{cn}\subset XHP_1\subset S^{O_{\cD,m,k}(cn)}$, we may assume that \eqref{eq:XHP.proper} holds. Moreover, if we increase $k$ if necessary by an amount depending only on $\cD$ then \eqref{eq:XHP.proper} and \eqref{eq:prog.inverse} combine with Lemma \ref{lem:coset.reps.local} to imply that on replacing $X$ with the set $X'$ given by Lemma \ref{lem:coset.reps.local} we have \eqref{eq:X.in.S^C} for some $t=t_\cD$ and
\begin{equation}\label{eq:before.rho}
S^{cn}\subset XHP_1^{O_\cD(1)}\subset S^{O_{\cD,m,k}(cn)}.
\end{equation}
Lemma \ref{lem:lengths.powers} and \eqref{eq:lengths.powers} then imply that we may choose $\rho\ll_\cD1$ such that we may replace $P_1^{O_\cD(1)}$ in \eqref{eq:before.rho} by $P=P_\ord(u;\rho L)$, and Lemma \ref{lem:upper-tri.dilate} implies that $P$ is in $O_\cD(1)$-upper-triangular form. Note also that $HP$ is $m$-proper. Provided $c$ is small enough in terms of $\cD$, $m$ and $k$, we then have \eqref{eq:point.where.we.assume.minimality}, as required.

Given $u_i$ with $\zeta(i)=1$, note that the ordered coset progression
\[
HP_\ord(u_1,\ldots,u_{i-1},u_{i+1},\ldots,u_d;\rho L_1,\ldots,\rho L_{i-1},\rho L_{i+1},\ldots,\rho L_d)
\]
formed by deleting the generator $u_i$ from $P$ is still $m$-proper and in $O_\cD(1)$-upper-triangular form, and so we may assume that $u_i$ is necessary for \eqref{eq:point.where.we.assume.minimality} to hold in the sense that there exists $s\in S^{cn}$ and $x\in X$, and $p\in HP$ with non-zero $u_i$-coordinate such that $s=xp$, as required. Provided $n$ is large enough in terms of $\cD$, $m$ and $k$, applying Lemma \ref{lem:set.growth.in.terms.of.prog} with $q=cn/2$ then combines with \eqref{eq:X.in.S^C} and \eqref{eq:point.where.we.assume.minimality} to imply that $XHP^r\subset S^{rn}\subset XHP^{O_{\cD,m,k}(r)}$ for every $r\in\N$, as required.
\end{proof}

\section{Persistence of polynomial growth}\label{sec:conj-1}
In this section we prove Theorem \ref{thm:persitstence.poly.growth}.

\begin{lemma}\label{lem:weight-1.linear}Let $C>0$ and $d_0,k\in\N$. Then there exists $m\ll_{C,d_0,k}1$ such that if $P=P_\ord(u;L)$ is an $m$-proper ordered progression of rank $d\le d_0$ in $C$-upper-triangular form, and if $p,q\in P^k$ with $p=u_1^{p_1}\cdots u_d^{p_d}$ and $q=u_1^{q_1}\cdots u_d^{q_d}$, then we have $pq\in P_\ord(u;mL)$, and for every $i$ with $\zeta(i)=1$ the $u_i$-coordinate of $pq$ with respect to $P$ is $p_i+q_i$.
\end{lemma}
\begin{proof}
This follows from Lemma \ref{lem:lengths.powers} and repeated application of the upper-triangular form and the identity $vu=uv[v,u]$.
\end{proof}

\begin{proof}[Proof of Theorem \ref{thm:persitstence.poly.growth}]
Let $d_0\in\N$ the maximum possible rank $d$ given by applying Proposition \ref{prop:inverse+}, and let $C$ also take the maximum possible value it can take in the conclusion of that proposition, assuming $D$ is as in the theorem we are proving. Set
\begin{equation}\label{eq:k=10}
k=10,
\end{equation}
and let $m$ be the constant given by applying Lemma \ref{lem:weight-1.linear}, noting that we may assume that $m\ge1$. Finally, assume that $n\ge\max\{N,NM\}$ for $N=N_{\cD,m,k}$ as in Proposition \ref{prop:inverse+}, and let $X,H,P,c,t$ be as given by that result.

Since $m,t\ll_\cD1$, if $n$ is large enough in terms of $\cD$ then \eqref{eq:point.where.we.assume.minimality} shows that applying Lemma \ref{lem:set.growth.in.terms.of.prog} with $q=cn/2$ gives
\[
S^{4cn}\subset XHP^9.
\]
Applying it with $q=t$ also implies by \eqref{eq:point.where.we.assume.minimality} that $S^{rt}\subset X(HP\cap S^{3t})^r$ for every $r\in\N$, and in particular that
\[
S^{cn}\subset X(HP\cap S^{3t})^{\lceil cn/t\rceil}.
\]
Provided again that $n$ is large enough in terms of $\cD$, these two containments combine with \eqref{eq:X.in.S^C} to imply that
\begin{equation}\label{eq:XHP9}
S^{cn}\subset X(HP\cap S^{3t})^{\lceil cn/t\rceil}\subset XHP^9.
\end{equation}

We claim that, for every $j\le\lceil cn/t\rceil$, for every $q_1,\ldots,q_j\in HP\cap S^{3t}$ we have $q_1\cdots q_j\in HP^9$; \cref{lem:weight-1.linear} will then allow us to control the coordinates of the $q_i$ with respect to $P$. The claim is trivial for $j=1$, and for $j>1$ we may assume by induction that $q_1\cdots q_{j-1}\in HP^9$, and hence that $q_1\cdots q_j\in HP^{10}$. However, \eqref{eq:XHP9} implies that $q_1\cdots q_j\in XHP^9$, and so \eqref{eq:XHP.proper} and \eqref{eq:k=10} imply that in fact $q_1\cdots q_j\in HP^9$, and the claim is proved. By \eqref{eq:XHP9}, this implies in particular that for every $s\in S^{cn}$ there exist $q_1(s),\ldots,q_{\lceil cn/t\rceil}(s)\in HP\cap S^{3t}$ and $x(s)\in X$ such that
\[
s=x(s)q_1(s)\cdots q_{\lceil cn/t\rceil}(s)
\]
and
\begin{equation}\label{eq:bootstrap.2}
q_1(s)\cdots q_j(s)\in HP^9
\end{equation}
for every $j$.

We now claim that $L_i\gg_\cD n$ for every generator $u_i$ of $P$ with $\zeta(i)=1$. Proposition \ref{prop:inverse+} implies that for every such $u_i$ there exists $s_i\in S^{cn}$ and $x_i\in X$, and $p_i\in HP$ with non-zero $u_i$-coordinate, such that $s_i=x_ip_i$. This implies in particular that $x_ip_i=x(s_i)q_1(s_i)\cdots q_{\lceil cn/t\rceil}(s_i)$, and so \eqref{eq:XHP.proper}, \eqref{eq:k=10} and \eqref{eq:bootstrap.2} imply that $x(s_i)=x_i$ and, more importantly,
\[
p_i=q_1(s_i)\cdots q_{\lceil cn/t\rceil}(s_i).
\]
Since $p_i$ has non-zero $u_i$-coordinate, this combines with \eqref{eq:bootstrap.2} and Lemma \ref{lem:weight-1.linear} to imply that some $q_j(s_i)$ has non-zero $u_i$-coordinate. However, since $q_j(s_i)\in S^{3t}$, we have $q_j(s_i)^\ell\in S^{cn}$ for every $\ell\in\N$ with $1\le\ell\le cn/3t$. Since $S^{cn}\subset XHP$ and $q_j(s_i)\in HP$, it therefore follows from repeated application of \eqref{eq:XHP.proper} and \eqref{eq:k=10} that in fact $q_j(s_i)^\ell\in S^{cn}\cap HP$ for every $\ell\in\N$ with $1\le\ell\le cn/3t$, and so Lemma \ref{lem:weight-1.linear} implies that $L_i\gg_\cD n$, as claimed.

The upper-triangular form of $HP$ therefore implies that for every $i$ we have $L_i\gg_\cD n^{\zeta(i)}$. Writing $\omega=\sum_i\zeta(i)$, the inequality \eqref{eq:proper.size} and the properness of $HP$ therefore imply that $|HP|\gg_\cD n^\omega|H|$. Combined with \eqref{eq:persist.poly.growth} and the $r=1$ case of \eqref{eq:inverse}, this implies that there exists $a=a_\cD$ such that $n^\omega\le aMn^D$. It follows that if
\[
n>(aM)^\frac{1}{1-\{D\}}
\]
then $\omega<\lfloor D\rfloor+1$. Since each $\zeta(i)\in\Z$, this in fact gives
\begin{equation}\label{eq:D<d}
\omega\le\lfloor D\rfloor.
\end{equation}
On the other hand, Lemma \ref{lem:lengths.powers}, \eqref{eq:proper.size} and the properness of $HP$ imply that
\begin{equation}\label{eq:P.growth}
|HP^r|\ll_\cD r^\omega|HP|
\end{equation}
for every $r\in\N$, and so for every such $r$ we have
\begin{align*}
|S^{rn}|&\ll_\cD|HP^{O_\cD(r)}|&\text{(by \eqref{eq:inverse})}\\
    &\ll_\cD r^{\lfloor D\rfloor}|HP|&\text{(by \eqref{eq:D<d} and \eqref{eq:P.growth})}\\
    &\le r^{\lfloor D\rfloor}|S^n|&\text{(by \eqref{eq:inverse})},
\end{align*}
and the theorem is proved.
\end{proof}


\section*{Acknowledgments} 
It is a pleasure to thank Itai Benjamini and Emmanuel Breuillard for helpful conversations and comments. We are also indebted to an anonymous referee for an extremely detailed report, including a number of suggestions that have improved the presentation of the paper considerably, particularly in \cref{sec:gen.of.progs}.

\bibliographystyle{amsplain}


\begin{dajauthors}
\begin{authorinfo}[romain]
  Romain Tessera\\
  Laboratoire de Math\'ematiques d'Orsay\\
  Univ.~Paris-Sud\\
  CNRS\\
  Universit\'e Paris-Saclay\\
  91405 Orsay\\
  France\\
  tessera\imageat{}phare\imagedot{}normalesup\imagedot{}org\\
  \url{http://www.normalesup.org/~tessera/}
\end{authorinfo}
\begin{authorinfo}[matt]
  Matthew Tointon\\
  Pembroke College\\
  Cambridge\\
  CB2 1RF\\
  United Kingdom\\
  mcht2\imageat{}cam\imagedot{}ac\imagedot{}uk\\
  \url{https://tointon.neocities.org/}
\end{authorinfo}
\end{dajauthors}

\end{document}